\documentclass[12pt]{article}
\usepackage{amsmath,amssymb,amsthm
,makeidx,fancyhdr
}
\usepackage{color}
\usepackage[all]{xy}
\usepackage[latin1]{inputenc}
\usepackage{tikz}

\theoremstyle{definition}
\newtheorem{theorem}{Theorem}[section]
\newtheorem{proposition}[theorem]{Proposition}

\newtheorem{claim}[theorem]{Claim}
\newtheorem{subclaim}{subclaim}[theorem]
\newtheorem{lemma}[theorem]{Lemma}
\newtheorem{corollary}[theorem]{Corollary}

\newtheorem{remark}[theorem]{Remark}

\newtheorem{definition}[theorem]{Definition}

\newcommand{\mo}{\triangleleft}
\newcommand{\Ult}{\operatorname{Ult}}

\newcommand{\fr}{{}^\frown}

\newcommand{\name}{\dot}
\newcommand{\can}{\check}
\newcommand{\force}{\Vdash}

\newcommand{\la}{\langle}
\newcommand{\ra}{\rangle}

\newcommand{\uhr}{\upharpoonright}

\newcommand{\power}{\mathcal{P}}

\newcommand{\po}{\mathcal{P}}
\newcommand{\qo}{\name{{Q}}}

\newcommand{\U}{\mathcal{U}}

\newcommand{\K}{\mathcal{K}}
\newcommand{\isom}{\simeq}

\newcommand{\pzero}{\mathcal{P}^0}
\newcommand{\onto}{\twoheadrightarrow}
\renewcommand{\i}{j^1}

\newcommand{\zerohandgrenade}{0^{\: \mid^{\! \! \! \bullet}}}

\renewcommand{\pzero}{\mathcal{P}^0}
\newcommand{\pone}{\mathcal{P}^1}
\newcommand{\ptwo}{\mathcal{P}^2}
\newcommand{\qzero}{\mathcal{Q}^0}
\newcommand{\qone}{\mathcal{Q}^1}
\newcommand{\qtwo}{\mathcal{Q}^2}

\newcommand{\Gzero}{G^0}
\newcommand{\Gone}{G^1}
\newcommand{\Gtwo}{G^2}

\renewcommand{\c}{c}
\newcommand{\T}{T}
\newcommand{\F}{\mathcal{F}}

\newcommand{\pX}{\mathcal{P}^X}

\newcommand{\GX}{G^X}
\newcommand{\UX}[2]{U^X_{(#1,#2)}}

\DeclareMathOperator{\Sacks}{Sacks}
\DeclareMathOperator{\Code}{Code}
\DeclareMathOperator{\Coll}{Coll}
\DeclareMathOperator{\supp}{supp}
\DeclareMathOperator{\On}{On}
\DeclareMathOperator{\cp}{cp}
\DeclareMathOperator{\len}{len}
\DeclareMathOperator{\Cf}{Cof}

\DeclareMathOperator{\dom}{dom}

\DeclareMathOperator{\rank}{rank}

\DeclareMathOperator{\GCH}{GCH}

\DeclareMathOperator{\tamerank}{Trank}

\renewcommand{\u}{\bar{u}}

\newcommand{\red}{\textcolor[rgb]{1.00,0.00,0.00}}
\newcommand{\blue}{\textcolor[rgb]{0.00,0.00,1.00}}

\title{The structure of the Mitchell order - II}
\author{Omer Ben-Neria 
\footnote{The paper is a part of the author Ph.D. written in Tel-Aviv University under the supervision of
Professor Moti Gitik.}}
\date{{}}

\begin{document}
\maketitle
\begin{abstract}
We address the  question regarding the structure of the Mitchell order
on normal measures. We show that every well founded order can be realized as the Mitchell 
order on a measurable cardinal $\kappa$, from some large cardinal assumption.
\end{abstract}

\section{Introduction}\label{Section - Introduction}
In this paper we address the question regarding the possible structure
of the Mitchell order $\mo$, at a measurable cardinal $\kappa$.
The Mitchell order was introduced by William Mitchell in \cite{Mitchell - MO}, who showed it is a well-founded order. The question whether every
well-founded order can be realized as the Mitchell order on the
set of some specific measurable cardinal, has been open since. 
By combining various forcing techniques with inner model 
theory we succeed in constructing models which realize every well-founded
order as the Mitchell order on a measurable cardinal.

Given two normal measures
$U,W$, we write $U \mo W$ to denote that
$U \in M_W \cong \Ult(V,W)$. 
For every measurable cardinal $\kappa$, let $\mo(\kappa)$ be  
the restriction of $\mo$ to the set of normal measures on $\kappa$,
and let $o(\kappa) = \rank(\mo(\kappa))$  be its (well-foundedness) rank.

The research on the possible structure on the Mitchell order $\mo(\kappa)$ is closely related to the question of its possible size, namely, 
the number of normal measures on $\kappa$: The first results by Kunen \cite{Kunen} and by Kunen and Paris \cite{Kunen-Paris} showed that this number 
can take the extremal values of $1$ and $\kappa^{++}$ (in a model of $\GCH$) respectively.
Soon after, Mitchell \cite{Mitchell - MO} \cite{Mitchell - MO-rev} showed that this size 
can be any cardinal $\lambda$ between $1$ and $\kappa^{++}$, under the large cardinal assumption and in a model of $o(\kappa) = \lambda$. 
Baldwin \cite{Baldwin} showed that for $\lambda < \kappa$ and from stronger large cardinal assumptions, $\kappa$ can also be the first measurable cardinal.
Apter-Cummings-Hamkins \cite{Apter Cummings Hamkins} proved that there can be $\kappa^+$ normal measures on $\kappa$ from the minimal assumption 
of a single measurable cardinal; for $\lambda < \kappa^+$, Leaning \cite{Leaning} reduced the large cardinal assumption from $o(\kappa) = \lambda$ 
to an assumption weaker than $o(\kappa) = 2$. The question of the possible number of normal measures on $\kappa$ was finally resolved by 
Friedman and Magidor in \cite{Friedman-Magidor - Normal Measures}, were it is shown that $\kappa$ can carry any number of normal measures 
$1 \leq \lambda \leq \kappa^{++}$ from the minimal assumption. The Friedman-Magidor method will be extensively used in this work.

Further results where obtained on the possible structure of the Mitchell order: Mitchell \cite{Mitchell - MO} and Baldwin \cite{Baldwin} showed that 
from some large cardinal assumptions, every well-order and pre-well-order (respectively) can be isomorphic to $\mo(\kappa)$ at some $\kappa$. 
Cummings \cite{Cummings I},\cite{Cummings II}, and Witzany \cite{Witzany} studied the $\mo$ ordering in various generic extensions, 
and showed that $\mo(\kappa)$ can have a rich structure.  
Cummings constructed models where $\mo(\kappa)$ embeds every order from a specific family of orders we call tame. Witzany showed that in a Kunen-Paris extension of a Mitchell model $L[\mathcal{U}]$, with $o^{\U}(\kappa) = \kappa^{++}$, every well-founded order of cardinality $\leq \kappa^+$ embeds into $\mo(\kappa)$. \\

In this paper, the main idea for realizing well-founded orders as $\mo(\kappa)$ is to force over an extender model $V = L[E]$ with a sufficiently rich $\mo$ structure on a set of extenders at $\kappa$. 
 {By forcing over $V$ we can collapse the generators of these extenders, giving rise to extensions of these extenders, which are equivalent to ultrafilters on $\kappa$.} 
The possible structure of the Mitchell order on arbitrary extenders was previously studied by Steel \cite{Steel - MO} and Neeman \cite{Neeman - MO} who showed that the well-foundedness of the Mitchell order fails exactly at the level of a rank-to-rank extenders.

For the most part, the extenders on $\kappa$ which will be used, do not belong to the main sequence $E$. Rather, they are of the form $F' = i_\theta(F)$, where $F \in E$ overlaps a measure on a cardinal $\theta > \kappa$, and $i_\theta$ is an elementary embedding with $\cp(i_\theta) = \theta$.  
There is a problem though; the extenders $F'$ may not be $\kappa-$complete.
To solve this we incorporate an additional forcing extension by which $F'$
will generically regain its missing sequence of generators\footnote{I.e., while $M_{F'} \cong \Ult(V,F')$ may not be closed under $\kappa-$sequences, its embedding $j_{F'} : V \to M_{F'}$ will extend in $V[G]$ to $j' : V[G] \to M_{F'}[G']$ so that ${}^\kappa M_{F'}[G'] \subset M_{F'}[G']$.}.\\
Forcing the above would translate the $\mo$ on certain extenders $F'$, to $\mo(\kappa)$; however some of the normal measures in the resulting model will be unnecessary and will need to be destroyed.  
This will be possible since the new normal measures on $\kappa$ are separated by sets\footnote{Namely,
we can associate to each normal measure $U$ on $\kappa$ a set 
$X_U \subset \kappa$, which is not contained in any distinct normal measure.} 
 allowing us to remove the undesired normal measures in an additional generic extension, which we refer to as a \emph{final cut}. \\
 
The combination of these methods will be used to prove the main result:
\begin{theorem}\label{MO II - main theorem}
Let $V = L[E]$ be a core model. Suppose that $\kappa$ is a cardinal in $V$ and $(S,<_S)$ is a well-founded order of cardinality $\leq \kappa$, so that
  \begin{enumerate}
  \item there are at least $|S|$ measurable cardinals above $\kappa$; let $\theta$ be the supremum of the successors of the first $|S|$,
  \item there is a $\mo-$increasing sequence of $(\theta+1)-$ strong extenders $\vec{F} = \la F_\alpha \mid \alpha < \small{\rank(S,<_S)}\ra$
  \end{enumerate}
  Then there is a generic extension $V^*$ of $V$ in which $\mo(\kappa)^{V^*} \cong (S,<_S)$.
 \end{theorem}

\noindent In particular, if $E$ contains a $\mo-$increasing sequence of extenders $\vec{F} = \la F_\alpha \mid \alpha < \kappa^+\ra$, so that each $F_\alpha$ overlaps the first $\kappa$ measurable cardinals above $\kappa$, then 
  every well founded order $(S,<_S) \in V$ of cardinality $\leq \kappa$ can be realized as $\mo(\kappa)$ in a generic extension of $V$.
   As an immediate corollary of the proof we have that under slightly stronger large cardinal assumptions, including a class of cardinals $\kappa$
  carrying similar overlapping extenders,  there is a class forcing extension $V^*$ in which every well founded order $(S,<_S)$ is isomorphic to 
  $\mo(\kappa)$ for some $\kappa$.\\
 
\noindent This paper is the second of a two-parts study on $\mo$. In the first part  
 \cite{OBN - Mitchell order I}, a wide family of well-founded orders named tame orders, was isolated  and it was shown that every tame order of cardinality at most $\kappa$ can be realized from an assumption weaker than $o(\kappa) = \kappa^+$.\\
The principle characteristic of tame orders is that they do not embed two specific orders: $R_{2,2}$ on a set of four elements, and $S_{\omega,2}$ on a countable set.

In section \ref{Section - II - First Overlapping} we consider  
a first example in which $V = L[E]$ contains a $\mo-$increasing sequence of extenders in $\kappa$, overlapping a single measure on a  cardinal $\theta > \kappa$. We use some specific $\mo-$configurations in $L[E]$ to produce models
which realize $R_{2,2}$ and $S_{\omega,2}$ as $\mo(\kappa)$. 
We also construct a model in which $o(\kappa) = \omega$ but there
is no $\omega$ sequence in $\mo(\kappa)$. In Section \ref{II Section - maintheorem} 
we extend our overlapping framework to models containing extenders $\vec{F} = \la F_\alpha\mid \alpha < l(\vec{F})\ra$ which overlap a 
sequence of measurable cardinals $\vec{\theta} = \la \theta_i \mid i < \omega\ra$. 
It is in this framework where we first need to deal with non-complete extenders of the form $F_{\alpha,c} = i_c(F_\alpha)$, where $i_c$ result from iterated ultrapowers by measures on the cardinals in $\vec{\theta}$. We introduce a poset designed to force the completeness of (extensions of) $F_{\alpha,c}$, 
and combine our methods to prove the main theorem. 
Finally, in {Section \ref{Section - II - Open Problems} we list further questions.}

For the most part, the notations in this paper continue the conventions in \cite{OBN - Mitchell order I}. Note that we use the Jerusalem convention for the forcing order, in which $p \geq q$ means that the condition $p$ is stronger than $q$. For every $(\kappa,\lambda)-$extender $F$, and $\gamma < \lambda$, 
we denote the $\gamma-$th measure in $F$ by $F(\gamma)$.

\section{A First Use of Overlapping Extenders}\label{Section - II - First Overlapping}

In this section we consider the the Mitchell order in an extender model $V = L[E]$  
in the sense of  \cite{Steel - HB}, 
which is minimal with respect to a certain large cardinal property.

\noindent An extender $F$ on the coherent sequence $E$ has a critical point $\cp(F)$,
and support \[\nu(F) = \sup(\{\kappa^+\} \cup \{\xi + 1 \mid \xi \text{ is a generator of } F\}) .\]
$\xi$ is a generator of $F$ if there are no $a \in [\xi]^{<\omega}$ and $f \in V$
so that $\xi = [f,a]_F$.
The index $\alpha = \alpha(F)$ of $F$ on the main sequence $E$ (i.e., $F = E_\alpha$) is given by 
$\alpha = (\nu(F)^+)^{\Ult(F,L[E\uhr\alpha])}$

Our large cardinal assumptions of $V = L[E]$ include the following requirements:
\begin{enumerate}
\item There are measurable cardinals $\kappa < \theta$ where $\theta$ is the first measurable above $\kappa$.
\item There is a $\mo-$increasing sequence $\vec{F} = \la F_\alpha \mid \alpha < l(\vec{F})\ra$, $l(\vec{F}) < \theta$, of $(\kappa,\theta^{++})-$extenders. 
\item Each $F_\alpha$ is $(\theta+2)-$strong, i.e., $V_{\theta+2} \subset \Ult(V,F_\alpha)$.
\item $\vec{F}$ consists of all full $(\kappa,\theta^{++})-$extenders on the main sequence $E$. 
\item There are no stronger extenders on $\kappa$ in $E$ ($o(\kappa) = \theta^{++} + l(\vec{F})$).
\item There are no extenders $F \in E$ so that $\cp(F) < \kappa$ and $\nu(F) \geq \kappa$.
\end{enumerate}

\begin{remark}
\begin{enumerate}
\item
The overlapping assumptions described here are in the realm of almost-linear iterations (i.e., weaker than zero hand-grenade, $\zerohandgrenade$). The existence of the core model for such large cardinal assumptions are established in \cite{Schindler - ICM}.

\item $U$ is equivalent to an extender on the main sequence $E$, where $\nu(U) = \theta^+$, thus $\alpha(U) < {\theta^{++}}^V$. In particular, the trivial completion of $U$ appears before $F_0 \in \vec{F}$ on the main sequence $E$. This fact will be used in the proof of Proposition \ref{Proposition - II - MO ordering by P}, where we coiterate certain ultrapowers of $L[E]$ by iterating the least disagreement.
\end{enumerate}
\end{remark}

\begin{definition}[$i_n$, $\theta_n$, $F_{\alpha,n}$,$\Theta$]${}$
\begin{enumerate}
 \item For every $\alpha < l(\vec{F})$ let $j_{F_{\alpha}}: V \to M_{F_\alpha} \cong \Ult(V,F_\alpha)$
 be the induced ultrapower embedding.
 We point out that the fact $F_\alpha$ is $(\theta+2)-$strong guarantees that $U \in M_{F_{\alpha}}$.
 
 \item
  For every $n < \omega$, let $i_n : V \to N_n \cong \Ult^n(V,U)$
  be the $n-$th ultrapower embedding by $U$ and
  $\theta_n = i_n(\theta)$. Clearly, $\theta_n$ is the first measurable cardinal above $\kappa$ in $N_n$. 
  
  \item For every $\alpha < \len(\vec{F})$ and $n < \omega$, define $F_{\alpha,n} = i_n(F_\alpha)$.
  $F_{\alpha,n}$ is a $(\kappa,{\theta^{++}}^V)$ extender on $\kappa$ in $N_n$ since both $\kappa$ and ${\theta^{++}}^V$ are fixed points of $i_n$. \\
  
  $F_{\alpha,n}$ is clearly a $\kappa-$complete extender on $\kappa$ in $N_n$. Since $V$ and $N_n$ have the same subsets of $\kappa$ it follows that $F_{\alpha,n}$ is an extender of $V$ as well. 
  
  \item For every $\nu < \kappa$, let $\Theta(\nu)$ be the first
  measurable cardinal above $\nu$ for every $\nu< \kappa$. Since $\Theta(\nu)$ is not a limit of measurable cardinals it  carries a unique normal measure denoted by $U_{\Theta(\nu)}$.
\end{enumerate}
\end{definition}

\begin{lemma}
For every $\alpha < l(\vec{F})$ and $n < \omega$, $F_{\alpha,n}$ is the $(\kappa,\theta^{++})-$extender derived in $V$ from the
composition $i_n^{M_{F_\alpha}} \circ j_{F_\alpha} : V \to M_{\alpha,n}$
where $M_{\alpha,n} \cong \Ult^{(n)}(M_{F_{\alpha}},U)$.
\end{lemma}

\begin{proof}
For simplicity of notations, we prove the Lemma for $n = 1$. The proof for arbitrary $n < \omega$ is similar. 
Here, $i_1^{M_{F_{\alpha}}}: M_{F_{\alpha}} \to N_{\alpha,1}$ is the $1-$st ultrapower embedding of $M_{F_{\alpha}}$ by $U$.
Let $j_{\alpha,1}  = i_{1}^{M_{F_{\alpha}}} \circ j_{F_{\alpha}} : V \to M_{\alpha,1}$ denote this composition. 
We need to verify that for every $\gamma < \theta^{++}$ and $X \subset \kappa$, 
$X \in F_{\alpha,1}(\gamma)$ if and only if $\gamma \in j_{\alpha,1}(X)$,
where $F_{\alpha,1}(\gamma)$ is the $\gamma-$th measure in $F_{\alpha,1}$.
As an element in $N_{1} \cong \Ult(V, U)$, $\gamma$ is represented as $\gamma = i_1(g)(\theta)$ where $g: \theta \to \theta^{++}$ is a function in $V$.
Note that $F_{\alpha,1} = i_{1}(F_{\alpha}) \in N_1$ and $X = i_1(X) \in N_1$, so
$X \in F_{\alpha,1}(\gamma)$ if and only if $i_1(X) \in i_1(F_{\alpha})(\gamma) = i_1(F_{\alpha})(i_1(g)(\theta))$.
By L\'{o}s Theorem, the last is equivalent to $\{\nu < \theta \mid X \in F_{\alpha}(g(\nu))\} \in U$, thus
\[ \gamma \in j_{\alpha,1}(X) \Longleftrightarrow \{\nu < \theta \mid X \in F_{\alpha}(g(\nu))\} \in U.\]
Consider again $j_{\alpha,1} = i_1^{M_{F_\alpha}} \circ j_{F_\alpha}$. 
Note that every $f : \theta \to \theta^{++}$ in $V$ can be easily coded as a subset of $\theta \times \theta^+ \times \theta^+$ and therefore belongs to $V_{\theta+2} \subset M_{F_\alpha}$.
It follows that $\gamma = i_1^{M_{F_\alpha}}(g)(\theta)$ is represented by the same function $g \in N_{\alpha,1} \cong \Ult(M_{F_\alpha},U)$, and hence
\[
\gamma \in i_{\alpha,1}(X) \iff i_1^{M_{F_\alpha}}(g)(\theta) \in i_1^{M_{F_\alpha}}(j_\alpha(F))(i_1^{M_{F_\alpha}}(X))
 \iff
\{\nu < \theta \mid g(\nu) \in j_{F_\alpha}(X)\} \in U.\]
The claim follows as $g(\nu) \in j_{F_\alpha}(X)$ if and only if $X \in F_\alpha(g(\nu))$. 
\end{proof}

\noindent The fact $l(\vec{F}) < \theta$ implies that the generators of each $j_{\alpha,n},M_{\alpha,n}$ belong to $[\kappa,\theta^{++})$\footnote{I.e., for every $x \in M_{\alpha,n}$ there is $\gamma \in [\kappa,\theta^{++})$ and $f : \kappa \to V$ so that $x = j_{\alpha,n}(f)(\gamma)$}. We conclude the following

\begin{corollary}
For every $\alpha < l(\vec{F})$ and $n < \omega$ we have
\begin{enumerate}
\item 
The ultrapower of $V$ by $F_{\alpha,n}$ is given by
\[j_{\alpha,n} = i_n^{M_{F_\alpha}} \circ j_{F_\alpha} : V \to M_{\alpha,n} \cong \Ult(V,F_{\alpha,n})\]

\item $F_{\alpha,n}$ is a $\kappa-$complete extender in $V$, i.e., 
$M_{\alpha,n}$ is closed under $\kappa-$sequences form $V$.

\item  $j_{\alpha,n}(\Theta)(\kappa) = \theta_n$ and  $U_{j_{\alpha,n}(\Theta)(\kappa)} = i_n(U)$.
\end{enumerate}
\end{corollary}

\subsection{Collapsing and Coding}\label{SubSection - Collapsing and Coding}

Let $\nu \leq \kappa$ and suppose that $V^*$ is a set generic extension of $V = \K(V^*)$ which preserves all stationary subsets of $\nu^+$. 
For every $A \subset \nu^+$ in $V^*$, we define a coding poset $\Code(\nu^+,A)$ to be forced over $V^*$.
 $V$. 

\begin{definition}[$\Code(\nu^+,A)$]\label{Definition - Collapse and Code}
Let $\vec{S} = \la S_i \mid i < \nu^+ \ra$ be the $<_{\K(V^*)}$ minimal $\Diamond_{\nu^+}$ 
sequence in $V = \K(V^*)$. This sequence is definable from $H(\nu^+)^{V}$.
Let $\la T_i \mid i < \nu^+\ra$ be a sequence of 
pairwise disjoint stationary subsets of  $\Cf(\nu) \cap \nu^+$, defined by $T_i = \{ \mu \in \Cf(\nu) \cap \nu^+ \mid S_\mu = \{i\}\}$. \\
Conditions $c \in \Code(\nu^+,A)$ are closed bounded subsets $c \subset \nu^+$ 
which satisfy the following properties for every $i < \nu$:
\begin{itemize}
 \item $c \cap T_{4i} = \emptyset$ if $i \in A$, and
\item $c \cap T_{4i + 1} = \emptyset$ if $i \not\in A$.
\end{itemize}
For every $c,d \in \Code(\nu^+,A)$, $d \geq c$ if 
\begin{itemize}
\item $d$ is an end extension of $c$,
 \item $(d \setminus c) \cap T_{4 i + 2} = \emptyset$ if $i \in c$, and 
 \item $(d \setminus c) \cap T_{4 i + 3} = \emptyset$ if $i < \max(c)$, $i \not\in c$.
 \end{itemize}
\end{definition}

\noindent The forcing $\Code(\nu^+,A)$ is $<\nu^+-$distributive.
Let $C \subset \nu^+$ be a $\Code(\nu^+,A)-$ generic club over $V^*$.
It is clear from the definition of $\Code(\nu^+,A)$ that for every $i < \nu^+$, $C$ is almost disjoint from $T_{4i}$ if $i \in A$; from 
$T_{4i+1}$ if $i \not\in A$; from 
$T_{4i+2}$ if $i \in C$; from 
$T_{4i+3}$ if $i \not\in C$.
The sets $T_j$, $j < \nu^+$, which are not required to be almost disjoint from $C$, remain stationary (see \cite{Friedman-Magidor - Normal Measures}).
Therefore, for every $i < \nu^+$, exactly one of  $T_{4i}$,$T_{4i+1}$ is not stationary in $V^*$, and exactly one of  $T_{4i+2}$,$T_{4i+3}$ is not stationary in $V^*$.

Let $g : \nu^+ \to \tau$ be a surjection from $\nu^+$ onto some ordinal $\tau$, and 
let $A_g \subset \nu^+$ be a canonical encoding of $g$
defined by $\la \mu_0, \mu_1 \ra \in A_g$ if and only if  $g(\nu_0) \leq g(\nu_1)$. Here $\la \cdot , \cdot \ra$ is the Godel pairing function $\nu^+ \times \nu^+$. 
Note that $g$ can be easily reconstructed from $A_g$. We define $\Code(\nu^+,g)$ to be $\Code(\nu^+,A_g)$.

\begin{definition}[Collapsing and Coding Iteration $\po$]\label{Definition - Code,Collapse, and iteration}
Define an iteration  $\po = \po_{\kappa+1} = \la \po_\nu, \qo_\nu \mid \nu \leq \kappa\ra$ over $V = L[E]$.
We use the Friedman-Magidor non-stationary support, i.e., every $p \in \po_\nu$ belongs to the inverse limit of the posets $\la \po_\mu \mid \mu < \nu\ra$
with the restriction that if $\nu$ is inaccessible, then the set of $\mu<\nu$ such that $p_\mu$
is nontrivial, is a nonstationary subset of $\nu$.
For every $\nu \leq \kappa$ if $\nu$ is not an inaccessible limit of measurable cardinals 
then $\force_{\po_\nu} \qo_\nu = \emptyset$.
Otherwise, $\force_{\po_\nu} \qo^1_\nu = \Coll(\nu^+,\Theta(\nu)^{++})*\Code(\nu^+,\name{g_\nu})$ where
\begin{enumerate}
\item $\Coll(\nu^+,\Theta(\nu)^{++})$ is the Levy collapsing forcing, introducing a surjection
$g_\nu : \Theta(\nu)^{++} \onto \nu^+$, and 
\item $\Code(\nu^+,\name{g_\nu})$ is the coding 
of the generic collapsing function.
\end{enumerate}
\end{definition}

\begin{remark}\label{Remark - I - NonStationary Support}
\begin{enumerate}
$\qo_\nu$ is $<\nu-$closed  and $<\nu^+-$distributively closed for every $\nu < \kappa$.
The uniqueness of the $\Code(\nu^+)$ generic at every non-trivial stage $\nu \leq \kappa$ implies that for every $V-$generic filter $G \subset \po$, $G$ is the unique $\po$ generic over $V$ in $V[G]$.

\item 
The iteration $\po$ is a variant of the Friedman-Magidor iteration (\cite{Friedman-Magidor - Normal Measures}
) where the collapsing posets $\Coll(\nu^+,\Theta(\nu)^{++})$ are replaced with generalized Sacks forcings.
\end{enumerate}
\end{remark}

In order to force with $\Code(\nu^+,\name{g_\nu})$ over a $\po\uhr\nu * \Coll(\nu^+,\Theta(\nu)^{++})-$generic extension of $V$, 
it is necessary that the extensions preserves 
the stationary subsets of $\nu^+$ in $V$, which are used in $\Code(\nu^+,\name{g_\nu})$. This is indeed the case. The proof is similar to the one given by Friedman-Magidor. 

The next results follow from the arguments by Friedman and Magidor\footnote{I.e., 
Lemma 5 and Lemma 14 in \cite{Friedman-Magidor - Normal Measures}}:
\begin{lemma}\label{Lemma - II - basic results on po}
\begin{enumerate}
 \item For every inaccessible limit of measurable cardinals $\nu \leq \kappa$, the poset $\po\uhr \nu * \Coll(\nu^+,\Theta(\nu)^{++})$
 preserve stationary subset of $\nu^+$.
 \item The iteration $\po$ is $\sigma-$closed and does not collapse inaccessible limits of measurable cardinals in $V$. 
 \item For every function $\phi : \kappa \to \kappa$ in a generic extension of $V$ by $\po$, 
 there exists some $f : \kappa \to [\kappa]^{<\kappa}$ in $V$, so that
 $|f(\nu)| \leq \Theta^{++}(\nu)$, and $\phi(\nu) \in f(\nu)$ for every $\nu < \kappa$. 
\end{enumerate}
\end{lemma}

We proceed to study the normal measures on $\kappa$ in a $\po$ generic extension. Let $G \subset \po$ be generic over $V$.
\begin{lemma}\label{Lemma - II - po genericity extension}
For every $\alpha < l(\vec{F})$ and $n < \omega$, there is a unique $M_{\alpha,n}-$generic filter
$G_{\alpha,n} \subset j_{\alpha,n}(\po)$ so that $j_{\alpha,n}``G \subset G_{\alpha,n}$.
\end{lemma}
\begin{proof}
We have that
$j_{\alpha,n}(\po)\uhr(\kappa+1) = \po$. 
The coding generics and the fact $H(\kappa^+)^{M_{\alpha,n}} = H(\kappa^+)^{V}$ imply
that $G \in V[G]$ is the only possible $M_{\alpha,n}-$generic filter for $\po$.
We prove that $\{j_{\alpha,n}(q')\setminus (\kappa+1) \mid q' \in G\}$ generates
a  $M_{\alpha,n}[G]-$generic filter for $j_{\alpha,n}(\po)\setminus (\kappa+1)$. 
The proof consists of two parts: In the first part (\textbf{1}) we restrict ourselves to conditions  $q \in G\uhr\kappa \subset \po_\kappa$ and $j_{\alpha,n}(q) \in j_{\alpha,n}(\po_\kappa) = j_{\alpha,n}(\po)\uhr j_{\alpha,n}(\kappa)$ (i.e., we do not address the last forcing step $j_{\alpha,n}(\qo_\kappa)$). We show that for every $\po-$name of a dense open set $D \subset j_{\alpha,n}(\po_\kappa)\setminus (\kappa+1)$ there is $q \in G\uhr\kappa$ such that $j_{\alpha,n}(q)\uhr (\kappa+1) \force j_{\alpha,n}(q)\setminus (\kappa+1) \in D$. Then, in the second part (\textbf{2}),  we build on the results of the first part and deal with dense open sets $E \subset j_{\alpha,n}(\po)\setminus (\kappa+1)$. \\

\textbf{1. }
Let $D$ be a $\po-$name for a dense open set of
$j_{\alpha,n}(\po_\kappa)\setminus (\kappa+1)$.
There are $\gamma < \theta^{++}$ and $f : \kappa \to V$ so that
$D = j_{\alpha,n}(f)(\gamma)$. We may assume that for every $\mu < \kappa$, if $\kappa(\mu) \leq \mu$ is the largest cardinal up to $\mu$ which is a limit of measurable cardinals, then
$f(\mu)$ is a $\po\uhr (\kappa(\mu)+1)$ name for a dense open set of $\po_\kappa\setminus(\kappa(\mu)+1)$\\

Let $p$ be a condition in $\po_\kappa$. Working in $V$, we simultaneously define three sequences:
\begin{enumerate}
\item $\la p_i \mid i \leq \kappa\ra$, an increasing sequence of conditions above $p$. 
\item  $\la C_i \mid i \leq \kappa\ra$, a sequence of closed unbounded subset of $\kappa$. 
\item  $\vec{\nu} = \la \nu_i \mid i < \kappa\ra$, a continuous increasing sequence of ordinals 
 below $\kappa$. 
 \end{enumerate}
 
We take $p_0$ to be an extension of $p$ which belongs to $f(0)$, $C_0 \subset \kappa$ to be a closed unbounded set, disjoint
from $\supp(p_0)$, and $\nu_0 = \min(\{ \nu \in C_0 \mid \nu \text{ is inaccessible and limit of measurable cardinals}\})$. 
The forcing $\po_\kappa \setminus (\nu_0 + 1)$ 
is $\Theta(\nu_0)^{+3}-$closed and therefore the intersection 
$\bigcap \{f(\mu)\mid \mu \leq \Theta(\nu_0)^{++}\}$ is dense in $\po_\kappa \setminus (\nu_0 + 1)$.
Let $p_1$ be an extension of $p_0$ so that
\begin{enumerate}
 \item $p_1 \uhr \nu_0 + 1 = p_0 \uhr \nu_0 + 1$, and
 \item $p_1\uhr \nu_0 + 1 \force p_1\setminus \nu_0+1 \in 
 \bigcap\{f(\mu) \mid \mu \leq \Theta(\nu_0)^{++}\}$.
\end{enumerate}
We take $C_1\subset C_0$ to be a closed unbounded set so that $C_1 \cap \supp(p_1) = \emptyset$ and  $C_1 \cap (\nu_0+1) = C_0 \cap (\nu_0+1)$. We then choose 
$\nu_1 = \min( \{ \nu \in C_1 \setminus (\nu_0 + 1)\mid  \nu \text{ is inaccessible and limit of measurable cardinals}\})$.\\
\noindent Suppose that the three sequences have been constructed up to some $i^* \leq \kappa$, and satisfy the following properties:
\begin{enumerate}
 \item $\la p_j \mid j < i^*\ra$ is an increasing sequence of conditions and $\la \nu_j \mid j < i^*\ra$ is a continuous increasing sequence of ordinals below $\kappa$, 
 \item $p_{j_1} \uhr \nu_{j_1} + 1 = p_{j_2} \uhr \nu_{j_1} + 1$ and $C_{j_2} \subset C_{j_1}$ for every $j_1 < j_2 < i^*$,
 \item $p_{j+1} \uhr \nu_{j} + 1 \force p_{j+1}\setminus \nu_j+1 \in 
 \bigcap\{f(\mu) \mid \mu \leq \Theta(\nu_j)^{++}\}$ whenever $j+1 < i^*$,
 \item $C_{j} \cap \supp(p_{j}) = \emptyset$ for every $j < i^*$,
 \item $\{\nu_j \mid j < i^*\}  \subset \bigcap_{i < i^*}C_{i}$ if $i^* < \kappa$.
\end{enumerate}
If $i^* = i+1$ is a successor ordinal then we define 
$p_{i+1},C_{i+1},\nu_{i+1}$ from $p_i,C_i,\nu_i$ the same way 
$p_{1},C_{1},\nu_{1}$ were defined from $p_0,C_0,\nu_0$.\\
Suppose $i^* = \delta$ is a limit ordinal. 
Let $\nu_\delta = \cup_{i<\delta}\nu_i$ and define $C_\delta$ to be $\bigcap_{i<\delta}C_i$ if $\delta < \kappa$, and 
$\triangle^{\vec{\nu}}C_i = \{ \alpha < \kappa \mid \forall i < \kappa. \thinspace (\nu_i < \alpha) \rightarrow (\alpha \in C_i)\}$ if $\delta = \kappa$.  
For every $j \leq \delta$ we have that $\{\nu_i \mid i < \delta\} \subset C_j$ and therefore $\nu_\delta \in C_j$. 

Let us define $p_\delta$. 
We first construct $p_\delta\uhr\nu_\delta$: 
For every $i < \nu_\delta$ let $(p_\delta)_i = (p_{j})_{i}$ where $j < \delta$ is such that
$i < \nu_j$. It follows that $p_\delta\uhr \nu_j + 1 = p_j\uhr \nu_j +1$ for every 
$j < \delta$.
Let us verify that $p_\delta\uhr \nu_\delta$ has
a nonstationary support.
For every inaccessible cardinal $\gamma < \nu_\delta$ we have that $p_\delta\uhr \gamma = p_\gamma\uhr \gamma$ hence $\supp(p_\delta\uhr\nu_\delta) \cap \gamma$ is nonstationary in $\gamma$. 
If $\nu_\delta$ is inaccessible (i.e., $\nu_\delta = \delta$) then $\supp(p_\delta\uhr \nu_\delta) \subset \nu_\delta$ is nonstationary in $\nu_\delta$ as it is disjoint $\{\nu_i \mid i < \delta\}$.\\
If $\delta = \kappa$ then $\nu_\delta = \kappa$, so $p_\delta = p_\delta\uhr\nu_\delta$ and we are done. 
Suppose that $\delta < \kappa$. Set $(p_\delta)_{\nu_\delta} = 0_{\qo_{\nu_\delta}}$ and let $p_\delta \setminus (\nu_\delta+1)$ be
a common extension of $\{ p_i\setminus \nu_\delta +1 \mid i < \delta\}$
so that $\supp(p_\delta \setminus \nu_\delta+1) = \bigcup_{i < \delta}\supp(p_i \setminus \nu_\delta+1)$. This is clearly possible as $\po_{\delta}\setminus (\nu_\delta+1)$
is $\delta^+-$closed, and it is not difficult to see that $p_\delta$, $C_\delta$, and $\nu_\delta$ satisfy the inductive assumptions.\\
This concludes the construction of the sequence $\la p_i \mid i \leq \kappa\ra$.
Let $q = p_\kappa$. For every $i < \kappa$ we have that $q \uhr \nu_i+1 \force q \setminus (\nu_i + 1) \in \bigcap\{f(\mu) \mid \mu \leq \Theta(\nu_i)^{++}\}$. 
Let $j_{n,\alpha}(\{\nu_i \mid i < \kappa\}) = \{\nu_i \mid i < j_{n,\alpha}(\kappa)\}$. Since $\{\nu_i \mid i < \kappa\}$ is a closed unbounded set in $\kappa$ we conclude that $\nu_\kappa = \kappa$ and that
\[j_{\alpha,n}(q)\uhr (\kappa+1) \force j_{\alpha,n}(q)\setminus(\kappa+1) \in 
\bigcap\{j_{\alpha,n}(f)(\mu) \mid  \mu \leq \theta_n^{++}\}.\]
The last intersection includes $D = j_{\alpha,n}(f)(\gamma)$.\\

\textbf{2. }
Suppose now that $E$ is dense open set of $j_{\alpha,n}(\po)\setminus (\kappa+1)$.
There are $e : \kappa \to V$ and $\gamma < \theta^{++}$ such that $E = j_{\alpha,n}(e)(\gamma)$. 
We may assume that for every $\mu < \kappa$, $e(\mu)$ is a $\po_{\kappa(\mu)+1}$ name
of a dense open set of $\po\setminus (\kappa(\mu)+1)$. 
Let $E(\name{G(j_{\alpha,n}(\po_\kappa)})$ be a $j_{\alpha,n}(\po_\kappa)$ for the set of all conditions $e \in j_{\alpha,n}(\qo_\kappa)$ for which there is some $t$ in the generic filter $G(j_{\alpha,n}(\po_\kappa))$ such that $t \fr \name{e} \in E$. 
$E(\name{G(j_{\alpha,n}(\po_\kappa)})$ is dense open in $j_{\alpha,n}(\qo_\kappa)$.
Similarly, for every $\mu < \kappa$ $e(\mu)(G(\po_{\kappa}))$ is a name for a dense open set
of $\qo_\kappa$.  
Fix a condition $p' = \la p'_\alpha \mid \alpha \leq \kappa\ra \in \po$. The fact $\qo_\kappa$ is
$<\kappa^+-$distributive implies there is a $\po_\kappa$ name $q'_\kappa$ so that
$p' \uhr \kappa \force p'_\kappa \leq q'_\kappa \in \bigcap_{\mu<\kappa} e(\mu)(G(\po_\kappa))$.
Let $p = p'\uhr \kappa$. We have that  $j_{\alpha,n}(p) \force j_{\alpha,n}(q_\kappa') \in E(\name{G(j_{\alpha,n}(\po_\kappa)})$. 
Let $D \subset j_{\alpha,n}(\po_\kappa)\setminus (\kappa+1)$ be a $\po-$name for the set of all  $t' \in j_{\alpha,n}(\po_\kappa)\setminus (\kappa+1)$ for which $t' \fr j_{\alpha,n}(q'_\kappa) \in E$. 
$D$ is a dense open set of $j_{\alpha,n}(\po_\kappa)\setminus (\kappa+1)$ so by the construction in the first part \text{(1)}, there is a condition $q \geq p$ in $\po_\kappa$ such 
that $j_{\alpha,n}(q)\uhr (\kappa+1) \force j_{\alpha,n}(q)\setminus (\kappa+1) \in D$. 
Taking $q' = q \fr q'_\kappa$, we conclude that $q' \geq p'$ and $j_{\alpha,n}(q')\uhr (\kappa+1) \force j_{\alpha,n}(q')\setminus (\kappa+1) \in E$. 
\end{proof}

\noindent \begin{definition}
Define the following in $V[G]$:
\begin{enumerate} 
\item
Let $G_{\alpha,n} \subset j_{\alpha,n}(\po)$ denote the unique generic filter over $M_{\alpha,n}$
with $j_{\alpha,n}``G \subset G_{\alpha,n}$, whose existence proved in Lemma
\ref{Lemma - II - po genericity extension}. 

\item 
Let $i_{\alpha,n} : V[G] \to M_{\alpha,n}[G_{\alpha,n}]$ be the unique extension of $j_{\alpha,n}$ to an embedding in $V[G]$. $i_{\alpha,n}$ is defined by $i_{\alpha,n}(\name{x}_{G}) = j_{\alpha,n}(\name{x})_{G_{\alpha,n}}$.  

\item 
Let $U_{\alpha,n}$ be the normal measure defined by
\[ U_{\alpha,n} = \{X \subset \kappa \mid \kappa \in i_{\alpha,n}(X)\}.\]
\end{enumerate}
\end{definition}

\begin{remark}\label{Remark - I - j``G generates U}
The proof of Lemma \ref{Lemma - II - po genericity extension} 
implies that for every $X  = \name{X}_G \subset  \kappa$ in $V[G]$, $X \in U_{\alpha,n}$ if and only if there  
exists a condition $p \in G$ so that in $M_{\alpha,n}[G]$,
\[j_{\alpha,n}(p)\setminus(\kappa+1) \force_{j_{\alpha,n}(\po)\setminus(\kappa+1)} \can{\kappa} \in j_{\alpha,n}(\name{X}).\] 
Equivalently, $X\in U_{\alpha,n}$ if and only if there is $p \in G$ so that 
\[ p \fr (j_{\alpha,n}(p)\setminus \kappa+1) \force_{j_{\alpha,n}(\po)} \can{\kappa} \in \name{X}.\] 
\end{remark}
${}$

Let
$\vec{\rho} = \la \rho_\zeta \mid \zeta < \kappa^+\ra$ 
be a sequence of canonical functions at $\kappa$, so that each $\rho_\zeta \in {}^\kappa\kappa$ has Galvin-Hajnal degree $\zeta$.
We have that 
\begin{enumerate}
\item for every $\zeta_0 < \zeta_1 < \kappa^+$, the set
$\{\nu < \kappa \mid \rho_{\zeta_0}(\nu)  \geq \rho_{\zeta_1}(\nu) \}$ is bounded in $\kappa$,
\item for every elementary embedding $j : V \to M$ in $V$, 
 then $\zeta = j(\rho_\zeta)(\kappa)$ for all $\zeta < \kappa^+$.  
\end{enumerate}

Let $\vec{g} = \la g_\nu \mid \nu < \kappa\ra$ be the generic sequence of collapsing functions induced from $G$,
i.e., $g_\nu : \nu^+ \onto \Theta(\nu)^{++}$ for every $\nu \leq \kappa$ which is an inaccessible limit of measurbale cardinals.\\

We can use each $g_\nu: \nu^+ \onto \Theta(\nu)^{++}$ to construct
a surjection $h_\nu : \nu^+ \twoheadrightarrow \Theta^{++}(\nu)$, which is canonically defined from $g_\nu$\footnote{For example, restrict the domain of  $g_\nu$ to ordinals $\mu < \nu^+$
so that $\mu = \min(g_\nu^{-1}(\{\tau\}))$ for some $\tau < \Theta^{++}(\nu)$, 
and then collapsing the restricted domain to achieve the bijection $h_\nu$.}. 
For every $\delta < \theta^{++}$ let $\zeta_\delta$ be the unique $\zeta < \kappa^+$ so that $\delta = h_\kappa(\zeta_\delta)$.
\begin{definition}\label{Definition - II - psi functions}
Let $\delta < \theta^{++}$. Define a function $\psi_\delta: \kappa \to \kappa$ by
$\psi_\delta(\nu) = h_\nu(\rho_\zeta(\nu))$ for every $\nu < \kappa$. 
\end{definition}

It follows that for every $\delta_0,\delta_1 < \theta^{++}$, the set
 $\{\nu < \kappa \mid \psi_{\delta_0}(\nu) \neq \psi_{\delta_1}(\nu)\}$ is bounded in $\kappa$. 
Also, for every $\alpha < l(\vec{F})$ and $n < \omega$, it is easy to check that $i_{\alpha,n}(\psi_\delta)(\kappa) = \delta$ for every $\delta < \theta$.

\begin{lemma}\label{Lemma - II - Extension Coincides with Measure Ultrapower}
$i_{\alpha,n} : V[G] \to M_{\alpha,n}[G_{\alpha,n}]$ coincides with the
ultrapower embedding of $V[G]$ by $U_{n,\alpha}$.
\end{lemma}
\begin{proof}
Since $U_{\alpha,n}$ is derived from $i_{\alpha,n}$, it is sufficient to verify that every $x \in M_{\alpha,n}[G_{\alpha,n}]$ can be represented as $x = i_{\alpha,n}(\phi)(\kappa)$, for some $\phi: \kappa \to V[G]$.
Every $x \in M_{\alpha,n}[G_{\alpha,n}]$ is of the form 
$x = (\name{x})_{G_{\alpha,n}}$, where $\name{x} = j_{\alpha,n}(f)(\delta) = 
i_{\alpha,n}(f)(\delta)$,
for some $f: \kappa \to V$ in $V$ and $\delta < \lambda$.
We may assume that $f(\nu)$ is a $\po$ name, for every $\nu < \kappa$.  
Define $\phi : \kappa \to V[G]$ by $\phi(\nu)  = f(\psi_\delta(\nu))_{G}$.
It follows that $x = i_{\alpha,n}(\phi)(\kappa)$ as
 $\delta = i_{\alpha,n}(\psi_\delta)(\kappa)$ and $G_{\alpha,n}  = i_{\alpha,n}(G)$.
\end{proof}

\subsection{The Mitchell Order in $V[G]$}

We show that $U_{\alpha,n}$, $\alpha < l(\vec{F})$, $n < \omega$ are the only normal measures on $\kappa$ in $V[G]$, and that
$U_{\alpha',n'} \mo U_{\alpha,n}$ if and only if 
$\alpha' < \alpha$ and $n' \geq n$. 
To prove these results, we use Schindler's description (\cite{Schindler - ICM}) of ultrapower restrictions.

\begin{proposition}\label{Proposition - II - Analysis of Normal Measures by P}
Let $W$ be a normal measure on $\kappa$ in $V[G]$.
There exists some $\alpha < l(\vec{F})$ and $n < \omega$, 
such that $W = U_{\alpha,n}$.
\end{proposition}
\begin{proof}
Let $j_W : V[G] \to M_W  \cong \Ult(V[G],W)$ be the induced
ultrapower embedding and $j = j_W\uhr V: V \to M$ be its restriction to $V = \K(V[G])$.
According to Schindler (\cite{Schindler - ICM}) there is an iteration tree $T$ on $V$
and a cofinal branch $b$ so that $\pi^{T}_{0,b} = j_W\uhr V : V \to M$ results from the normal iteration of $T$ along $b$.
Furthermore $M_W = M[G_W]$ where $G_W = j_W(G) \subset j(\po)$ is $M$-generic.
In \ref{Definition - II - psi functions} we defined a sequence of functions $\la \psi_\delta \mid \delta < \lambda\ra$ so that 
so that $\{\nu < \kappa \mid \psi_{\delta_0}(\nu) = \psi_{\delta_1}(\nu)\}$ is bounded in $\kappa$ for every distinct $\delta_0,\delta_1 < \theta$.
It follows that $j_W(\psi_{\delta_0})(\kappa) \neq j_W(\psi_{\delta_1})(\kappa)$ for every $\delta_0 \neq \delta_1$, thus $j_W(\kappa) \geq \theta^{++}$.  
Furthermore, as $\psi_\delta(\nu) < \Theta^{++}(\nu)$ for every $\delta < \theta^{++}$ and $\nu < \kappa$, it follows that $j_W(\psi_\delta)(\kappa) < j_W(\Theta^{++})(\kappa) = j(\Theta^{++})(\kappa)$
thus $\theta^{++} \leq j(\Theta^{++})(\kappa)$.
As $M_W = M[G_W]$ and $G_W \subset j(\po)$ is generic, $(\theta^{++})^V$ is collapsed to $\kappa^+$ in 
$M_W$. With this in mind, let us consider the iteration tree $T$ and the cofinal branch $b$ inducing $\pi^T_{0,b} = j$.
We first claim that the iteration of $b$ does not use the
same extender more the finitely many times. Otherwise there would be an ordinal
$\delta \leq  \pi_0(\kappa)$ which is the the limit of the critical points of an $\omega-$subiteration by the same measure. In particular $\Cf^{V[G]}(\delta) = \omega$.
Since $\delta$ is the image of 
the critical points it must be inaccessible in $M$. Furthermore, its cofinality in $M[G_W]$ is $\omega$ because $M[G_W] = M_W$ is closed under $\omega$ sequence in $V[G]$. It follows that the $M-$ generic set $G_W \subset j(\po)$ introduces a cofinal $\omega-$sequence in $\delta$. However this is impossible as $j(\po)$ is a $\sigma-$closed forcing. \\
Since $\cp(\pi^T_{0,b}) = \cp(j_W) = \kappa$, $\pi^T_{0,b}$ factors into $\pi^T_{0,b} = \pi^T_{1,b} \circ j_F$, where $F$ is an extender on $\kappa$ in $V$ and $\cp(\pi^T_{1,b}) > \kappa$. By the elementarity of $j_F : V \to M_F$, $j_F(\kappa)$ is an
inaccessible limit of measurable cardinals in $M_F \cong \Ult(V,F)$ and there are no extenders $F' \in M_F$ for which $\cp(F') < j_F(\kappa)$ and $\alpha(F') \geq j_F(\kappa)$.
Since the rest of the iteration along $b$ can apply each extender finitely many times,
$j_F(\kappa)$ must be fixed point of $\pi^T_{1,b}$, and an inaccessible and limit of measurable cardinals in $M$. By Lemma \ref{Lemma - II - basic results on po}, $j_F(\kappa)$ is not collapsed
in $M_W = M[G_W]$. It follows that $j_F(\kappa) \geq \theta^{++}$ so  $\nu(F) \geq \theta^{++}$.
Since the normal measure $U$ on $\theta$ is the only measure which overlaps extenders
on $\kappa$, it follows that $F$ belongs to $N_\eta \cong \Ult^\eta(V,U)$ for some ordinal $\eta$. We claim that $\eta < \omega$.
Otherwise, if $\theta_\omega$ is the $\omega-$th image of $\theta = \cp(U)$ the $\theta_\omega$ is an inaccessible cardinal in both $N_\eta$ and $M_F \cong \Ult(V,F)$. 
This is impossible as $\Cf^{V[G]}(\theta_\omega) = \omega$ but  $j(\po)$ $\sigma-$closed.
We conclude that $\eta = n$ for some finite $n < \omega$, and as $\nu(F) \geq \theta^{++}$ we get that $F = F_{\alpha,n}$ for some $\alpha < \len(\vec{F})$. 
Finally, we verify that $W = U_{\alpha,n}$. We can rewrite he restriction $j_W \uhr V = j$ as $j = \pi_1 \circ j_{\alpha,n}$ where $\cp(\pi_1) > \kappa$.
We clearly have that $j(\po)\uhr(\kappa+1) = \po$. The coding posets in $\po$ and the fact $H(\kappa^+)^{M_W} = H(\kappa^+)^V$ imply that $G = G_W\uhr (\kappa+1)$. 
Since we also have that $j``G \subset G_W$ we conclude that $p \fr (j(p)\setminus \kappa+1) \in G_W$ for every $p \in G$.
Suppose that $X = \name{X}_G \in U_{\alpha,n}$. 
According to  Remark \ref{Remark - I - j``G generates U}, 
there is some $p \in G$ so that
$p \fr (j_{\alpha,n}(p)\setminus \kappa+1) \force \can{\kappa} \in j_{\alpha,n}(\name{X})$.
By applying $\pi_1$ we conclude that $p \fr (j(p) \setminus \kappa+1) \force \can{\kappa} \in j(\name{X})$, thus $X \in W$. 
\end{proof}

\begin{proposition}\label{Proposition - II - MO ordering by P}
For every $\alpha,\alpha' < \len(\vec{F})$  and $n,n' < \omega$, 
$U_{\alpha',n'} \mo U_{\alpha,n}$ if and only if  $\alpha' < \alpha$ and $n' \geq n$. 
\end{proposition}
\begin{proof}
Suppose first that $n' \geq n$ and $\alpha' < \alpha$.
It is clear that $F_{\alpha',n'} \in M_{\alpha,n}$ and  that
$\po^{V} = \po^{M_{\alpha,n}}$. 
The construction of $U_{\alpha',n'}$ requires $F_{\alpha',n'}$,$V_{\kappa+1}$, and $G$; all belong to $M_{\alpha,n}[G_{\alpha,n}] = \Ult(V[G],U_{\alpha,n})$, thus $U_{\alpha',n'} \mo U_{\alpha,n}$. \\

Suppose now that $U_{\alpha',n'} \mo U_{\alpha,n}$.
To simplify our notations, let us denote $W_{\alpha',n'}$ by $W'$ and $W_{\alpha,n}$ by $W$. Accordingly, let $j_{W'} : V[G] \to M' \cong \Ult(V[G],W')$ and $j_{W} : V[G] \to M \cong \Ult(V[G],W)$.
Note that $\K(M) = M_{\alpha,n}$ and $\K(M') = M_{\alpha',n'}$are both extender models. Let us denote $\K(M)$ by $L[E^{M}]$ and $\K(M')$ by $L[E^{M'}]$. 
Since $W' \mo W$, we can form an ultrapower by $W'$ in $M$. Let $i' : M \to N' \isom \Ult(M,W')$, and denote $\K(N')$ by $L[E^{N'}]$.
According to Schindler (\cite{Schindler - ICM}), $\K(N')$ results from a
normal iteration of $K(M)$. Namely, there is an iteration tree $T$ of $K(M)$, and a cofinal branch $b$,
so that $j_{W'}\uhr K(M) = \pi^{\K(M)}_{0,b} : \K(M) \to \K(N')$.
Moreover, the proof of Theorem 2.1 in \cite{Schindler - ICM} implies that iteration tree $T$ is the tree which results 
from a comparison between $\K(M)$ and $\K(N')$.
Thus, when applying a comparison process to $L[E^{M}]$ and $L[E^{N'}]$
we get that $L[E^{N'}]$ does not move, and the iteration tree $T$ on $L[E^{M}]$,
is determined by comparing the extender sequence of the iterands of $L[E^{M}]$ with $E^{N'}$.\\
Now, $M$ is the ultapower of $V[G]$ by a normal measure on $\kappa$, so $M \cap H(\kappa^+) = V[G] \cap H(\kappa^+)$. 
Therefore, when taking the ultrapower of both model by $W'$, we get that 
$i' \uhr \kappa^+ = j_{W'}\uhr \kappa^+$ and that $M' \cap V[G]_{i'(\kappa)} = N' \cap V[G]_{i'(\kappa)}$. In particular, $E^{N'}\uhr i'(\kappa) = E^{M'}\uhr i'(\kappa)$. \\
It follows that when coiterating $\K(M) = L[E^M]$ with $\K(M') = L[E^{M'}]$,  the $L[E^{M'}]-$side does not move below $j_{W'}(\kappa)$, namely $E^{M'} \uhr j_{W'}(\kappa)$ is fixed throughout the comparison.\\
We now conclude that $\alpha' < \alpha$ and $n' \geq n$ by coiterating $\K(M) = M_{\alpha,n}$ with $\K(M') = M_{\alpha',n'}$.
Let $m = \min(n,n')$. If $n \neq n'$ then the first point of distinction between the $E^M$ and $E^{M'}$ is at the index of the measure $U^m = i_m(U)$, which is 
$\alpha = (\theta_m^{++})^{\Ult^{m+1}(V,U)} < \theta^{++} < j_{W'}(\kappa)$. 
We must have that $n' \geq n$, as otherwise the first steps in the coiteration would consists of ultrapowers by $i_{n'}(U) \in E^{M'}\uhr j_{W'}(\kappa)$ on the $L[E^{M'}]$ side. This contradicts the assumption that $E^{M'}\uhr j_{W'}(\kappa)$ remains fixed.
Let $k = n' - n \geq 0$. It follows that the first $k$ 
steps of the coiteration are coincides with the $k-$ultrapower of $L[E^M]$ by $i_n(U)$.
Let us denote the resulting ultrapower of $L[E^M]$ by $L[E^{M,k}]$. 
It is clear that the extenders sequences $E^{M'}$ and $E^{M,k}$
agree on indices up to $\theta^{++}$.
Therefore, the first possible difference between $E^{M'}$ and $E^{M,k}$ 
would be in extenders $F$ so that $\nu(F) = \theta^{++}$. These are
the extenders with support $\theta^{++}$ in the models
$M_{\alpha,n'}$ and $M_{\alpha',n'}$. 
We claim that  $\alpha' \neq \alpha$. Otherwise the agreement between $E^{M'}$ and $E^{M,k}$ would go above all extenders $F$ with critical point $\cp(F) =\kappa$. But this would imply that the hole iteration $T$ of $\K(M)$ is above $\kappa+1$ (Note that the critical points 
in the first $k$ steps of the iteration are above $\kappa$) and so 
$\cp(\pi^{\K(M)}_{0,b}) > \kappa$. This is absurd as $\pi^{\K(M)}_{0,b}  = i'\uhr \K(M)$ where $\cp(i') = \kappa$. We conclude that $\alpha' \neq \alpha$.
Finally, if $\alpha' > \alpha$ then 
the first disagreement between $E^{M'}$ and $E^{M,k}$ would be at the extender
$F_{\alpha',n'} \in E^{M'}\uhr j_{W'}(\kappa)$. This would contradict the fact that $E^{M'}\uhr j_{W'}(\kappa)$ remains fixed in the comparison. 
It follows that $\alpha' < \alpha$. 
\end{proof}

\subsection{A Final Cut}

We apply a final cut forcing over $V[G]$. 
For every $X \subset \kappa$ in $V[G]$, let $\pX$ be the final cut forcing by $X$, defined 
in \cite{OBN - Mitchell order I} (see Section 7).

\begin{lemma}\label{Lemma - II - Cut Down Measure Preservation}
Suppose that $X$ is a subset of $\kappa$ in $V[G]$ and let $\GX \subset \pX$ be generic over $V[G]$. 
For every normal measure $U$ on $\kappa$ in $V[G]$, if $X \not\in U$ then $U$ has a unique extension $U^X$ in  $V[G*\GX]$. Furthermore, these are the only normal measures on $\kappa$ in $V[G*\GX]$.
\end{lemma}
\begin{proof}
Suppose $U \in V[G]$ is a normal measure on $\kappa$ such that $X \not\in U$
and let $j : V[G] \to M[G_U] \cong \Ult(V[G],U)$ be its ultrapower embedding.
It is clear that $j(\pX)\uhr\kappa = \pX$, and stage $\kappa$ of
$j(\pX)$ is trivial as $\kappa \not\in j(X)$. 
The Friedman-Magidor iteration style implies that $j``\GX$ determines a unique generic filter $H^X \subset j(\pX)\setminus \kappa$ over $M[\GX]$. Setting $G^* = \GX * H^X$, we get that $G^* \subset j(\pX)$ is the unique generic filter
over $M$ for which $j``\GX \subset G^*$.  
It follows that
$j^* : V[G*\GX] \to M[G_U]$ is the unique extension of $j : V[G] \to M$ in $V[G*\GX]$ and that $U^X = \{Y \subset \kappa \mid \kappa \in j^*(Y)\}$ is the only normal measures extending  $U$ in $V[G*\GX]$.\\
Let $W$ be a normal measure on $\kappa$ in $V[G*\GX]$ and $j_W : V[G*\GX] \to M_W \cong \Ult(V[G*\GX],W)$. 
The embedding $j = j_W \uhr V : V \to M$ results from a normal iteration of $V$ given by a cofinal branch $b$ in an iteration tree $T$.
Furthermore, $M_W = M[G_W * \GX_W]$ where $G_W * \GX_W \subset j(\po * \pX)$ is generic over $M$.
It is easy to verify that the arguments 
Proposition \ref{Proposition - II - Analysis of Normal Measures by P} applies here as well. Note that the replacement of $\po$ with $\po * \pX$ does not affect the argument, as $\po*\pX$ (like $\po$) does not add new $\omega-$sequences and does not collapse inaccessible limits of measurable cardinals. 
It follows that there are $\alpha < l(\vec{F})$ and $n < \omega$ so that $U_{\alpha,n} \subset W$. \\
We claim that $X \not\in W$. Otherwise $\kappa \in j(X)$, so stage $\kappa$ in $j(\pX)$  is $\Code^*(\kappa)$.
Let $C \in M[G_W * \GX_W]$ be the $\Code^*(\kappa)$ generic club determined from $\GX_W$. We have that $C \cap X^\kappa_0 = \emptyset$ and that $C$ is closed unbounded in $V[G*\GX]$ (as $M[G_W * \GX_W] = M_W$ is closed under $\kappa-$sequences).
Thus $X^\kappa_0$ is nonstationary in $V[G*\GX]$, but this is impossible since $\pX$ preserves all stationary subset of $\kappa^+$.
We conclude that $W \cap V[G] = U_{\alpha,n}$ and that $X \not\in U_{\alpha,n}$.
Let us verify $U^X_{\alpha,n} \subseteq W$. Let $Y = (\name{Y})_{\GX}$ be a set in $U^X_{\alpha,n}$.
Since $j_{U_{\alpha,n}}``\GX$ generates a $j_{U_{\alpha,n}}(\pX)$ generic filter
over $M_{U_{\alpha,n}} \cong \Ult(V[G],U_{\alpha,n})$, it follows that there is a condition $p \in \GX$ so that 
$j_{U_{\alpha,n}}(p) \force \can{\kappa} \in j_{U_{\alpha,n}}(\name{Y})$.
Let $Y' = \{\nu < \kappa \mid p \force \can{\nu} \in \name{Y}\}$. 
It follows that $\kappa \in j_{U_{\alpha,n}}(Y')$. We conclude that $Y' \in U_{\alpha,n}$, a thus $Y' \in W$. But $Y' \subset Y$ because $p \in \GX$, 
so $Y \in W$. It follows that $U^X_{\alpha,n} = W$.
\end{proof}

A simple inspection of the proof of Proposition \ref{Proposition - II - MO ordering by P} shows that the Mitchell order on the set of normal measures $U^X_{\alpha,n}$ in $V[G*\GX]$ inherit the $\mo$ structure from $V[G]$.  
\begin{corollary}\label{Corollary - II - Mo after final cut}
$\UX{\alpha'}{n'} \mo \UX{\alpha}{n}$ if and only if  $\alpha' < \alpha$ and
 $n' \geq n$. 
\end{corollary}


\subsection{Applications}

We conclude this section with several applications, showing how to realize some new non-tame orders as $\mo(\kappa)$ in generic extensions of the form $V[G][\GX]$.

In part I (\cite{OBN - Mitchell order I}), we introduced a class of well founded order called tame orders,
and proved that every tame order $(S,<_S)$ of size at most $\kappa$, can be consistently realized as $\mo(\kappa)$.
An order $(S,<_S)$ is tame if and only if it does not contain two specific orders:
\begin{enumerate}
\item $(R_{2,2},<_{R_{2,2}})$ is an order on a set of four elements 
$R_{2,2} = \{x_0,x_1,y_0,y_1\}$, defined by 
$<_{R_{2,2}} = \{ ( x_0,y_0) , ( x_1,y_1 )\}$.

\item $(S_{\omega,2},<_{S_{\omega,2}})$ is an order on a disjoint union of two  countable sets
$S_{\omega,2} = \{x_n\}_{n < \omega} \uplus \{y_n \}_{n<\omega}$, defined by
$<_{S_{\omega,2}} = \{ (x_m, y_n) \mid m \geq n\}$. 
\end{enumerate}
Therefore $R_{2,2}$ and $S_{\omega,2}$ are the principal examples of orders which cannot be realized by the methods of Part I. Let us show that $R_{2,2}$ and $S_{\omega,2}$ can be realized in our new settings.\\

\textbf{First application - Realizing $S_{\omega,2}$}\\
Suppose that in $\len(\vec{F}) = 2$ in $V = L[E]$, i.e.,  $\vec{F} = \la F_0,F_1\ra$.
Propositions \ref{Proposition - II - Analysis of Normal Measures by P} and \ref{Proposition - II - MO ordering by P}
imply that the normal measures in $\kappa$ in $V[G]$ are given by
$\{ U_{\delta,n} \mid \delta < 2,  n < \omega\}$, where
$U_{\delta',m} \mo U_{\delta,n}$ if and only if  $\delta' = 0$, $\delta = 1$ and $m \geq n$. 
Therefore $\mo(\kappa)^{V[G]} \cong <_{S_{\omega,2}}$.

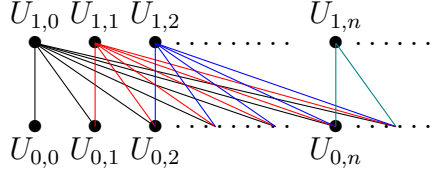
\begin{center}
\begin{tikzpicture}[xscale=0.8, yscale=0.37]

    \pgfmathtruncatemacro\mult{1} 
    
    \draw [step=1.0mm,thin,black]  node [] at (0,0) {$\bullet$};
    \draw [step=1.0mm,thin,black]  node [below] at (0,0) {$U_{0,0}$};
    
    \draw [step=1.0mm,thin,black]  node [] at (1,0) {$\bullet$};
    \draw [step=1.0mm,thin,black]  node [below] at (1,0) {$U_{0,1}$};
    
    \draw [step=1.0mm,thin,black]  node [] at (2,0) {$\bullet$};
    \draw [step=1.0mm,thin,black]  node [below] at (2,0) {$U_{0,2}$};
    
    \draw [step=1.0mm,thin,black]  node [] at (3,0) {$\dots\dots$};
    \draw [step=1.0mm,thin,black]  node [] at (4,0) {$\dots$};
    
    \draw [step=1.0mm,thin,black]  node [] at (5,0) {$\bullet$};
    \draw [step=1.0mm,thin,black]  node [below] at (5,0) {$U_{0,n}$};
    
        \draw [step=1.0mm,thin,black]  node [] at (6,0) {$\dots\dots$};
    
        \draw [step=1.0mm,thin,black]  node [] at (0,3) {$\bullet$};
    \draw [step=1.0mm,thin,black]  node [above] at (0,3) {$U_{1,0}$};
    
    \draw [step=1.0mm,thin,black]  node [] at (1,3) {$\bullet$};
    \draw [step=1.0mm,thin,black]  node [above] at (1,3) {$U_{1,1}$};
    
    \draw [step=1.0mm,thin,black]  node [] at (2,3) {$\bullet$};
    \draw [step=1.0mm,thin,black]  node [above] at (2,3) {$U_{1,2}$};
    
    \draw [step=1.0mm,thin,black]  node [] at (3,3) {$\dots\dots$};
    \draw [step=1.0mm,thin,black]  node [] at (4,3) {$\dots$};
    
    \draw [step=1.0mm,thin,black]  node [] at (5,3) {$\bullet$};
    \draw [step=1.0mm,thin,black]  node [above] at (5,3) {$U_{1,n}$};
    
   \draw [step=1.0mm,thin,black]  node [] at (6,3) {$\dots\dots$};

     \foreach \i in {0,...,6} {
        \draw [thin,black] (0,3) -- (\i,0) ;}

     \foreach \i in {1,...,6} {
        \draw [thin,red] (1,3) -- (\i,0) ;}
        
     \foreach \i in {2,...,6} {
        \draw [thin,blue] (2,3) -- (\i,0) ;}
        
     \foreach \i in {5,...,6} {
        \draw [thin,teal] (5,3) -- (\i,0) ;}

\end{tikzpicture}
\end{center}

\textbf{Second application - Realizing $R_{2,2}$}\\
Suppose that $l(\vec{F}) = 3$ in $V = L[E]$, i.e.  $\vec{F} = \la F_0,F_1,F_2\ra$.
The measures on $\kappa$ in $V[G]$ are given by 
in $\{U_{\delta,n} \mid \delta < 3,n < \omega\}$.
Let $S = \{U_{0,0},U_{1,0},U_{1,1},U_{2,1}\}$. 
The restriction of $\mo(\kappa)$ to $S$ includes the relations
$U_{0,0} \mo U_{1,0}$ and $U_{1,1} \mo U_{2,1}$,
therefore  $\mo(\kappa)^{V[G]}\uhr S \isom R_{2,2}$. 
Since there are only $\aleph_0$ many normal measures on $\kappa$
in $V[G]$, then the normal measures in $\{U_{\delta,n} \mid \delta < 3,n < \omega\}$
are separated by sets, and there is some $X \subset \kappa$ so that
$X \not\in U_{\delta,n}$ if and only if $U_{\delta,n} \in S$. 
By forcing with $\pX$ over $V[G]$ we get a model $V[G][\GX]$
in which  $\mo(\kappa)^{V[G][\GX]} \cong R_{2,2}$. 
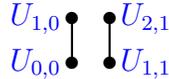
\newcommand*{\zc}{0}%
 \newcommand*{\oc}{1}%
\begin{center}
\begin{tikzpicture}[xscale=0.5, yscale=0.6]
    \draw  node [] at (\zc,\zc) {$\bullet$};
      \draw  node [left] at (\zc,\zc) {\blue{$U_{0,0}$}};
      \draw  node [] at (\zc,\oc) {$\bullet$};
      \draw  node [left] at (\zc,\oc) {\blue{$U_{1,0}$}};
      \draw  node [] at (\oc,\zc) {$\bullet$};
      \draw  node [right] at (\oc,\zc) {\blue{$U_{1,1}$}};
      \draw  node [] at (\oc,\oc) {$\bullet$};
      \draw  node [right] at (\oc,\oc) {\blue{$U_{2,1}$}};
      \draw [thick,black] (\zc,\zc) -- (\zc,\oc);      
      \draw [thick,black] (\oc,\zc) -- (\oc,\oc);            
\end{tikzpicture}
\end{center}

\textbf{Third application - $o(\kappa) = \omega$ but no $\omega-$increasing sequence in $\mo$}\\
Suppoes that $l(\vec{F}) = \omega$, i.e., $\vec{F} = \la F_i \mid i < \omega\ra$.
The normal measures on $\kappa$ in a $\po-$generic extension $V[G]$, are
of the form $U_{i,n}$ where $i < \omega$ and $n < \omega$. 
We define blocks of normal measures in $V[G]$:
\[B_n= \{U_{i,n} \mid k_n \leq i \leq k_n+n \}, \quad k_n = \frac{n(n+1)}{2}\]

We get that for every $n < \omega$:
\begin{enumerate}
 \item $|B_n| = n+1$, and
 \item The last $i-$index in a measure $U_{i,n} \in B_n$ is equal
 to the minimal $i-$index of a measure in $B_{n+1}$.  
\end{enumerate}
Proposition \ref{Proposition - II - MO ordering by P} implies
that for every $k < \omega$, $B_k$ is linerily ordered by $\mo(\kappa)$, so $\mo(\kappa)^{V[G]} \uhr B_k$ is 
$(k+1)-$increasing sequence in $\mo$. Moreover two measures from different blocks are $\mo$ incomparable.
The normal measures in $V[G]$ are separated by sets, so there is a set $X \subset \kappa$ which separates the measures in $S = \bigcup_{k<\omega} B_k$ from the of the rest
of the normal measures on $\kappa$. Let $\GX \subset \pX$ be generic over $V[G]$. 
It follows that $\mo(\kappa)^{V[G][\GX]}$ is isomorphic to a disjoint union of 
linear orders on $(k+1)-$elements, for every $k < \omega$. 
In particular $o^{V[G][\GX]}(\kappa) = \omega$, but there is no $\omega-$increasing sequence in $\mo(\kappa)$.

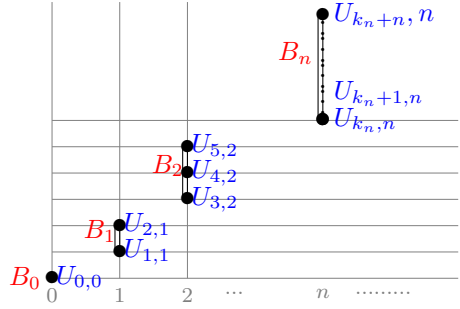
\newcommand*{\xMin}{0}%
\newcommand*{\xMidS}{2}
\newcommand*{\xMidE}{4}
\newcommand*{\xMax}{6}

\newcommand*{\yMin}{0}
\newcommand*{\yMidS}{6}%
\newcommand*{\yMidE}{6}
\newcommand*{\yMidEn}{10}
\newcommand*{\yMax}{10.5}%

\begin{center}
\begin{tikzpicture}[xscale=0.9, yscale=0.35]
    \pgfmathtruncatemacro\mult{1}

    \foreach \i in {\xMin,...,\xMidS} {
        \draw [very thin,gray] (\mult*\i,\yMin) -- (\mult*\i,\yMax) ;
        \draw [very thin,gray] node [below] at (\mult*\i,\yMin) {$\scriptstyle\i$};
    }
    \foreach \i in {\yMin,...,\yMidS} {
        \draw [very thin,gray] (\mult*\xMin,\i) -- (\mult*\xMax,\i) ;
    }

      \draw [step=1.0mm,thin,black]  node [] at (0,0) {$\bullet$};
       \draw [step=1.0mm,thin,black]  node [right] at (-0.1,0) {\footnotesize \blue{$U_{0,0}$}};
       \draw [step=1.0mm,thin,black]  node [left] at (0,0) {\footnotesize \red{$B_0$}};
	                             
      \foreach \x in {1,...,2} {
	  \pgfmathtruncatemacro\jstart{\x*(\x+1)/2}
	  \pgfmathtruncatemacro\jend{\x*(\x+1)/2 + \x - 1}
	  
	           \draw [step=1.0mm,thin,black]  node [right] at (\mult*\x-0.1,\jstart) {\footnotesize \blue{\footnotesize $U_{\jstart,\x}$}};
	           
	  \foreach \j in {\jstart,...,\jend} {
		  \pgfmathtruncatemacro\jnext{\j+1}
		  \draw [ step=1.0mm,thin,black] (\mult*\x,\j) -- (\mult*\x,\jnext) node [] at (\mult*\x,\jnext) {$\bullet$};
		  \draw [step=1.0mm,thin,black]  node [right] at (\mult*\x-0.1,\jnext) {\footnotesize\blue{$U_{\jnext,\x}$}};
	  }
	      \draw [step=1.0mm,thin,black]  node [] at (\mult*\x,\jstart) {$\bullet$};

	    \pgfmathtruncatemacro\jlast{\x*(\x+1)/2 + \x}  
	    \pgfmathtruncatemacro\xnext{\x+1}  
	    \draw [thin, black ,xshift=-2pt,]
		  (\mult*\x,\jstart) -- (\mult*\x,\jlast) node [black,midway,xshift=-0.2cm, yshift= 0.1cm]  
		  {\footnotesize \red{$B_{\x}$}};
	    }

      \draw [very thin,gray] node [below] at (\xMidE,\yMin) {$\scriptstyle{n}$};
      \draw [very thin,gray] node [below] at (\xMidE-1.3,\yMin) {${}_{\dots}$};
            \draw [very thin,gray] node [below] at (\xMidE+0.9,\yMin) {${}_{\dots\dots\dots}$};
      \draw [ step=1.0mm,thin,black] (\mult*\xMidE,\yMidE) -- (\mult*\xMidE,\yMidEn) ;
       \draw [step=1.0mm,thin,black]  node [] at (\mult*\xMidE,\yMidE) {$\bullet$};
              \draw [step=1.0mm,thin,black]  node [right] at (\mult*\xMidE,\yMidE) {\blue{\footnotesize $U_{k_{n},n}$}};
        \draw [step=1.0mm,thin,black]  node [] at (\mult*\xMidE,\yMidE ) {$\bullet$};
        \draw [step=1.0mm,thin,black]  node [right] at (\mult*\xMidE,\yMidE +1 ) {\blue{\footnotesize $U_{k_{n}+1,n}$}};
            \draw   node [] at (\mult*\xMidE,\yMidE + 1) {$\bf{\vdots}$};
                \draw   node [] at (\mult*\xMidE,\yMidE + 2) {$\bf{\vdots}$};
            \draw   node [] at (\mult*\xMidE,\yMidE + 3) {$\bf{\vdots}$};
                 \draw   node [] at (\mult*\xMidE,\yMidE + 4) {$\bf{\vdots}$};
            
            \draw   node [] at (\mult*\xMidE,\yMidE + 4) {$\bullet$};
             \draw [step=1.0mm,thin,black]  node [right] at (\mult*\xMidE,\yMidE + 4) {\blue{\footnotesize $U_{k_{n}+n},n$}};
            \draw [thin, black, xshift=-2pt,]
		  (\mult*\xMidE,\yMidE) -- (\mult*\xMidE,\yMidEn) node [black,midway,xshift=-0.3cm, yshift= 0.2cm] 
		  {\footnotesize $\red{B_n}$};
\end{tikzpicture}
\end{center}


\section{The Main Theorem}\label{II Section - maintheorem}
This section is devoted to proving the main Theorem (\ref{MO II - main theorem}). 
In subsection \ref{subSection - II - Realizing},
we first introduce a family of orders $(R^*_{\rho,\lambda}, <_{R^*_{\rho,\lambda}})$
$\rho,\lambda \in \On$, and show that every well-founded order $(S,<_S)$ embeds into $R^*_{\rho,\lambda}$ for $\rho = \rank(S,<_S)$ and $\lambda = |S|$. 
We then proceed to describe the revised ground model assumptions of $V = L[E]$ and introduce extenders $F_{\alpha,c}$ which are used to realize $R^*_{\rho,\lambda}$ using $\mo$. 
The forcing we apply to $V$ has two main goals: To generate extensions of $F_{\alpha,c}$ (in a generic extension) which are $\kappa-$complete, and to collapse their generators to $\kappa^+$, thus introducing equivalent normal measures on $\kappa$. \\
The generic extension $V^2$ of $V$ is obtained by a poset which includes three components: $\pzero$. $\pone$, and $\ptwo$.
\begin{enumerate}
\item $\pzero$ is the Friedman-Magidor forcing, splitting the measures and extenders on $\kappa$ in $V$ into $\lambda$ $\mo-$equivalent extensions. Much like the use of the Friedman-Magidor forcing in Part I (\cite{OBN - Mitchell order I}),
 the purpose of $\pzero$ is to allow simultaneously dealing with many different extenders in a single forcing. In subsection \ref{II subsection - P0} we describe certain extensions of the extenders $F_{\alpha,c} \in V$ in a $\pzero$ generic extension $V^0$, and introduce the key iterations and embeddings used in the subsequent extensions.

\item $\pone$ is a Magidor iteration of one-point Priky forcings.
The purpose of this poset is to introduce extensions of the extenders $F_{\alpha,c}$ which are $\kappa-$complete (i.e., extension of the $V-$iterated ultrapower by $F_{\alpha,c}$ which are closed under $\kappa-$sequences).
In subsection \ref{II subsection - P1} we describe a $\pone$ generic extension $V^1$. We then introduce and further investigate a collection of $\kappa$-complete extenders in $V^1$. 
\item $\ptwo$ is a collapse and coding iteration, similar to the collapsing and coding poset $\po$ introduced in the previous section. In subsection \ref{II subsection - P2} we describe the normal measures on $\kappa$ in $V^2 = (V^1)^{\ptwo}$  and $\mo(\kappa)$ in $V^2$, which embeds $R^*_{\rho,\lambda}$. 
\end{enumerate}
Finally, in subsection \ref{II subsection - PX} we apply a final cut extension
to form a model in which $\mo(\kappa) \cong (S,<_S)$.

\subsection{The Orders $(R^*_{\rho,\lambda},<_{R^*_{\rho,\lambda}})$}\label{subSection - II - Realizing}

\noindent For every ordinal $\lambda$, and  $c,c' : \lambda \to \On$, we write $c \geq c'$ when
$c(i) \geq c'(i)$ for every $i < \lambda$.

\begin{definition}[$R^*_{\rho,\lambda},<_{R^*_{\rho,\lambda}}$]${}$\\
For ordinals $\rho,\lambda$ let 
$R^*_{\alpha,\lambda} = \rho \times {}^\lambda 2$.
Let $<_{R^*_{\rho,\lambda}}$ be a order relation on $R^*_{\rho,\lambda}$ defined by
$(\rho_0,c_0) <_{R^*_{\alpha,\lambda}}(\rho_1,c_1)$
if and only if $\rho_0 < \rho_1$ and $c_0 \geq c_1$. 
\end{definition}

\begin{lemma}\label{II Lemma - embedding S in R^*}
Let $(S,<_S)$ be a well-founded order so that $|S| \leq \lambda$, and
$\rank(S,<_S) = \rho$, then $(S,<_S)$ can be embedded in $(R^*_{\rho,\lambda},<_{R^*_{\rho,\lambda}})$
\end{lemma}

\begin{proof}
Suppose that $(S,<_S)$ is a well founded order. For every $x \in S$
let $\u(x) = \{ y \in S \mid x <_S y \text{ or } y = x\}$,
and define a function $c'_x: S \to 2$ to be 
the characteristic function of $\u(x)$ in $S$,
i.e., for every $y \in S$,
\[
c'_x(y) = 
\begin{cases}
 1  &\mbox{if } y \in \u(x) \\
 0 &\mbox{otherwise. }
\end{cases}\]

Let $x,y$ be a distinct elements in $S$.
Note that if $x<_S y$ then $\u(y)  \subset \u(x)$ hence $c'_x \geq c'_y$, 
and if $x \not<_S y$ the $y \in \u(y) \setminus \u(x)$ so
$c'_x \not\geq c'_y$. \\
Let $\sigma : \lambda \to S$ be a bijection and define for every
$x \in S$, $c_x = c'_x \circ \sigma : \lambda \to 2$.
It follows that for every $x\neq y$ in $S$,
$c_x \geq c_y$ if and only if $x <_S y$.\\
We conclude that the function $\pi : S \to R^*_{\rho,\lambda}$ defined
by $\pi(x) = \la \rank_S(x), c_x\ra$, is an embedding of
$(S,<_S)$ into $(R^*_{\rho,\lambda},<_{R^*_{\rho,\lambda}})$.
\end{proof}

We conclude that the order $(S,<_S)$ is isomorphic to a restriction of $(R^*_{\rho,\lambda},<_{R^*_{\rho,\lambda}})$ to a subset of the domain.\\

\textbf{Let us assume from this point on that $S \subset R^*_{\rho,\lambda}$ and 
$<_S = <_{R^*_{\rho,\lambda}}\uhr S$.}\\

\begin{definition}[$S' = \la x_i \mid i < \lambda\ra$]\label{MO II - Definition - S'}${}$\\
Define $S' \subset \power(\lambda)$,
\[S' = \{x \subset \lambda \mid \text{ there is } (\alpha,c) \in S \text{ such that } x = c^{-1}(\{1\})\}.\]

\textbf{We fix a surjective enumeration $\la x_i \mid i < \lambda\ra$ of $S'$.}
\end{definition}

\noindent Suppose that  $V = L[E]$ is a core model which satisfy the following requirements:
\begin{enumerate}
\item There are $\lambda$ measurable cardinals above $\kappa$.
Let $\vec{\theta} = \la \theta_i \mid i < \lambda\ra$ is an increasing enumeration 
of the first and let $\theta = \bigcup_{i < \lambda}\theta_i^+$.
\item There is a $\mo-$increasing sequence $\vec{F} = \la F_\alpha \mid \alpha < l(\vec{F})\ra$, $l(\vec{F}) < \theta_0$, of $(\kappa,\theta^{+})-$extenders. 

\item 
Each $F_\alpha$ is $(\theta+1)-$strong.

\item $\vec{F}$ consists of all full $(\kappa,\theta^{+})-$extenders in $E$. 

\item
There are no stronger extenders on $\kappa$ in $E$ ($o(\kappa) = \theta^{+} + l(\vec{F})$).
\item There are no extenders $F \in E$ so that $\cp(F) < \kappa$ and $\nu(F) \geq \kappa$.
\end{enumerate}

For every $i < \lambda$, let $U_{\theta_i}$ be the unique normal measure
on $\theta_i$ in $V$.

\begin{definition}[$i_c$, $N_c$, $F_{\alpha,c}$, $M_{\alpha,c}$, $j_{\alpha,c}$]${}$
\begin{enumerate}
\item 
For any $c : \lambda \to 2$
let $i_{c} : V \to N_{c}$ be the elementary embedding formed by a linear iteration of the measures $U_{\theta_i}$ for which  $c(i) = 1$. 
We refer to this iteration as the $c-$derived iteration. 
\\
Therefore every set $x \in N_c$ is of the form $x = i_c(f)(\theta_{i_0},\dots,\theta_{i_{n-1}})$, where
$\{i_0,\dots,i_n\}$ is a finite subset of $c^{-1}(\{1\})$ and $f : \prod_{k<n}\theta_{i_k} \to V$ is a function in $V$. 

\item For every $i < \lambda$, let $\theta^{c(i)}_i = i_c(\theta_i)$.
Therefore $\theta^{c(i)}_i$ is the $i-$th measurable cardinal above
$\kappa$ in $N_c$. 

\item 
For every $\alpha < l(\vec{F})$ let $F_{\alpha,c} = i_c(F_\alpha)$. 
$F_{\alpha,c}$ is a $(\kappa,\theta^+)-$extender on $\kappa$ in $V$.

\item 
Note that $\la U_{\theta_{i}} \mid i < \lambda\ra$ belongs to $M_{\alpha}$.
Let $i_{\alpha,c} : M_{F_{\alpha}} \to M_{\alpha,c}$ 
be embedding which results from the  $c-$derived iteration of $M_{\alpha}$.

\item 
Let $j_{F_{\alpha,c}} : V \to M_{F_{\alpha,c}} \cong \Ult(V,F_{\alpha,c})$ denote the iterated ultrapower of $V$ by $F_{\alpha,c}$.
\end{enumerate}
\end{definition}

\begin{lemma}\label{Lemma - II - F_omega derived from i_omega circ j_F}
 For each $\alpha < l(\vec{F})$ and $c : \lambda \to 2$, $F_{\alpha,c}$ 
 is the $(\kappa,\theta^+)-$extender derived from the embedding $i_{\alpha,c} \circ j_{\alpha} : V \to M_{\alpha,c}$. Therefore $M_{F_{\alpha,c}} = M_{\alpha,c}$ and $j_{F_{\alpha,c}} = i_{\alpha,c} \circ j_{\alpha}$.
\end{lemma}

\begin{proof}
We need to verify that for every $\gamma < \theta^+$ and $Y \subset \kappa$, 
$Y \in F_{\alpha,c}(\gamma)$ if and only if
$\gamma \in i_{\alpha,c} \circ j_{\alpha}(Y)$.\\
We represent $\gamma$ as an element of the $c-$derived iteration if $M_{\alpha}$. Let  $i_0,\dots,i_{n-1} < \lambda$ be a finite sequence of indices,
and let $f : \prod_{k<n}\theta_{i_k} \to \theta^+$, 
so that $\gamma = i_{\alpha,c}(\theta_{i_0},\dots,\theta_{i_{n-1}})$.\\
Since $V_{\theta+1} \subset M_{\alpha}$ we get that 
$i_c \uhr V_{\theta+1} = i_c^{M_{\alpha}}\uhr V_{\theta+1}$.
In particular $\gamma = i_c(\theta_{i_0},\dots,\theta_{i_{n-1}})$.
As $F_{\alpha,c} = i_c(F_\alpha)$ we have that $Y \in F_{\alpha,c}(\gamma)$ if and only if
\begin{equation}\label{eq - 1}
\{\vec{\nu} \in \prod_{k<n}\theta_{i_k} \mid Y \in F_\alpha(f(\vec{\nu}))\} \in \prod_{i<n}U_{\theta_{i_k}}.
\end{equation}
Note that for every $\vec{\nu} \in \prod_{k<n}\theta_{i_k}$, $Y \in F_\alpha(f(\vec{\nu}))$ if and only
if $f(\vec{\nu})\in j_{\alpha}(Y)$. Therefore \ref{eq - 1} is equivalent to 
\begin{equation}\label{eq - 2}
\{\vec{\nu} \in \prod_{k<n}\theta_{i_k} \mid f(\vec{\nu})\in j_{\alpha}(Y)\} \in \prod_{k<n}U_{\theta_{i_k}}.
\end{equation}
The claim follows since \ref{eq - 2} can be seen as a statement of $M_{\alpha}$ which is equivalent to
$i^{M_F}_c(f(\theta_{i_0},\dots,\theta_{i_{n-1}})\in i_c\circ j_{F_\alpha}(Y)$. 
\end{proof}

\subsection{The Poset $\pzero$}\label{II subsection - P0}

\begin{definition}[$\Omega'$ and $\pzero$]${}$\\
Let $\Omega'$ denote the set of $\nu < \kappa$ which are inaccessible limit of measurable cardinals.\\

The poset $\pzero = \pzero_{\kappa+1} = \la \pzero_\nu, \qzero_\nu \mid \nu \leq \kappa\ra$ is a Friedman-Magidor poset (\cite{Friedman-Magidor - Normal Measures}) for splitting the $(\kappa,\theta^+)-$extenders in $V$ 
into $\lambda$ $\mo-$equivalent extensions. 
$\pzero$ is a nonstationary support iteration so that $\qo^0_\nu$ is not trivial for $\nu \in \Omega' \cup \{\kappa\}$,
and $\force_{\po^0_\nu} \qo^0_\nu = \Sacks_{\lambda}(\nu)*\Code(\nu)$, where
$\Sacks_{\lambda}(\nu)$ is a Sacks forcing with $\rho_\lambda(\nu)-$splittings, and $\Code(\nu)$  codes the generic Sacks function $s_\nu : \nu \to \lambda(\nu)$ and itself. Here 
\[ \lambda(\nu) = 
\begin{cases}
 \lambda  &\mbox{if } \lambda <  \kappa  \\
 \nu &\mbox{if } \lambda = \kappa
\end{cases}\] 
\end{definition}

Let $\Gzero \subset \pzero$ be a generic filter. We denote $V[\Gzero]$ by $V^0$. 
According to the analysis of Friedman and Magidor, for every $\alpha < l(\vec{F})$, the ultrapower embedding $j_{\alpha} : V \to M_{\alpha}\cong \Ult(V,{\alpha})$ has exactly $\lambda-$different extensions of the form $j^0_{\alpha,i} : V[\Gzero] \to  M_{\alpha}[\Gzero_{\alpha,i}]$, satisfying the following properties:
\begin{enumerate}
 \item $j^0_{\alpha,i}\uhr V = j_{\alpha}$,
 \item $j^0_{\alpha,i}(\Gzero) = \Gzero_{\alpha,i}$,
 \item $s^{\Gzero_{\alpha,i}}_{j_{\alpha}(\kappa)}(\kappa) = i$, 
	where $s^{\Gzero_{\alpha,i}}_{j_{\alpha}(\kappa)} : j_{\alpha}(\kappa) \to j_{\alpha}(\lambda)$ is the $\Gzero_{\alpha,i}-$generically derived Sacks function.
 \item  $V^0_{\theta+1} \subset M_{\alpha}[\Gzero_{\alpha,i}]$.
\end{enumerate}

\begin{definition}[$M^0_{\alpha,x}$, $j^0_{\alpha,x}$, $F^0_{\alpha,x}$, $\Omega'_x$]\label{Definition - Omega'x}${}$\\
Let $x \in S'$ and suppose that $\bf{x = x_i}$ in the enumeration of $S'$ introduced in
Definition \ref{MO II - Definition - S'}.
\begin{enumerate}
\item 
Let us denote $M_{\alpha}[\Gzero_{\alpha,i}]$ by $M^0_{\alpha,x}$, and
$j^0_{\alpha,i} : V[\Gzero] \to  M_{\alpha}[\Gzero_{\alpha,i}]$ by $j^0_{\alpha,x} : V[\Gzero] \to M^0_{\alpha,x}$. 

\item 
Let $F^0_{\alpha,x}$ denote the $(\kappa,\theta^+)-$extender in $V[\Gzero]$
derived from $j^0_{\alpha,x} : V^0  \to M^0_{\alpha,x}$.

\item 
Let $\Omega'_x = \{\nu \in \Omega' \mid s_{\nu} = s_\kappa\uhr\nu \text{ and } s_\kappa(\nu) = i\}$.
\end{enumerate}
\end{definition}

It easily follows that for every $x,y \in S'$,
$\Omega'_x \in F^0_{\alpha,y}(\kappa)$ if and only if $x = y$. \\
As $|\pzero| < \theta_0$ we have that for every $i  < \lambda$,
the normal measure $U_{\theta_i}$ in $V$ generates a unique normal measure on $\theta_i$ in $V^0 = V[\Gzero]$. Let us denote this extension by $U^0_{\theta_i}$. Note that $U^0_{\theta_i} \in M^0_{\alpha,x}$ for every $\alpha < \vec{F}$ and $x \in S'$.

\begin{definition}[$i^0_{\alpha,x,c}$, $N^0_{\alpha,x,c}$, $F^0_{\alpha,x,c}$]${}$
\begin{enumerate}
\item  Suppose that $\alpha < l(\vec{F})$, $x \in S'$, and $c : \lambda \to \omega$.
We define a linear iteration
\begin{equation}\label{Equation - II - F0 c gamma derived from a V0 iteration}
\la N^0_{\alpha,x,c,j}, i^0_{\alpha,x,c,j,j'} \mid j < j' \leq \lambda\ra 
\end{equation}
associated with $\alpha$, $x$, and $c$:
\begin{itemize}
\item $N^0_{\alpha,x,c,0} = M^0_{\alpha,x}$.
\item for every $j < \lambda$,
\[ i^0_{\alpha,x,c,j,j+1} : N^0_{\alpha,x,c,j} \to N^0_{\alpha,x,c,j+1} \cong \Ult^{\c(j)}(N^0_{\alpha,x,c,j}, U^0_{\theta_j})\]
is the $c(j)-$th iterated ultrapower embedding of $N^0_{\alpha,x,c,j}$ by $U^0_{\theta_j} = i^0_{\alpha,x,c,0,j}(U^0_{\theta_j})$.
\item 
for every limit ordinal $j' \leq \lambda$, 
$N^0_{\alpha,x,c,j'}$ is the direct limit of the iteration up to $j'$,
and $i^0_{\alpha,x,c,j,j'}$, $j < j'$, are the limit embeddings.

\item We denote $N^0_{\alpha,x,c,\lambda}$ by $N^0_{\alpha,x,c}$, 
and $i^0_{\alpha,x,c,0,\lambda}$ by $i^0_{\alpha,x,c} : M^0_{\alpha,x} \to N^0_{\alpha,x,c}$.
\end{itemize}
We refer to this iteration as the $c-$derived iteration of $M^0_{\alpha,x}$, and to $i^0_{\alpha,x,c} : M^0_{\alpha,x} \to N^0_{\alpha,x,c}$ as the $c-$derived embedding of $M^0_{\alpha,x}$.

\item Let $i^0_c : V^0 \to N^0_c$ be similarly defined $c-$derived iteration of $V^0$.

\item Let $F^0_{\alpha,c,x} = i^0_c(F_{\alpha,x})$. 
$F^0_{\alpha,c,x}$ is a $(\kappa,\theta^+)-$extender on $\kappa$ in $V^0$ as 
$\cp(i^0_c) > \kappa$  and $\theta^+ = i^0_c(\theta^+)$.
Moreover, the fact $V^0_{\theta+1} \subset M^0_{\alpha,x}$ implies that
$i^0_c\uhr \theta^+ = i^0_{\alpha,x,c} \uhr \theta^+$.
\end{enumerate}
\end{definition}

The proof of Lemma \ref{Lemma - II - F_omega derived from i_omega circ j_F}
can be easily modified to show the following
\begin{corollary}
For every $\alpha < l(\vec{F})$, $x \in S'$, and  $c : \lambda \to \omega$,  $F^0_{\alpha,x,c}$ is the $(\kappa,\theta^+)-$extender derived from $i^0_{\alpha,x,c} \circ j^0_{\alpha,x} : V^0 \to M^0_{\alpha,x,c}$, and therefore $M^0_{\alpha,x,c} \cong \Ult(V^0,F^0_{\alpha,x,c})$.
\end{corollary}

For every $\alpha < l(\vec{F})$, $c : \lambda \to \omega$, and $x,y \in S'$,
we saw that $\Omega'_x \in F^0_{\alpha,y}(\kappa)$ if and only if
$x = y$. Since $\kappa < \cp(i^0_c)$ it follows that the same is true for $\Omega'_x$ and $F^0_{\alpha,y,c}(\kappa)$.


\subsection{The poset $\pone$}\label{II subsection - P1}

\begin{definition}[$x-$suitable functions, $\Theta_j$, $\Theta$]${}$
\begin{enumerate}
\item 
Let $c_x : \lambda \to 2$ be the characteristic function of $x \subset \lambda$.
 We say that function $c : \lambda \to \omega$ is $x-$\textit{suitable} if and only if $c(j) = c_x(j)$  for all but finitely many $j < \lambda$. 
 
\item 
For every $j < \lambda$ and $\nu < \kappa$ let
$\Theta_j(\nu)$ to be the $j-$th measurable cardinal above $\nu$,
if $j < \nu$, and $0$ otherwise.\\
Define $\Theta : \kappa \to \kappa$ by $\Theta(\nu) = \bigcup_{j < \lambda}\Theta_j(\nu)^+$.
\end{enumerate}
\end{definition}
 
\begin{definition}[$\pone$]${}$\\
The poset  $\pone =  \la \pone_\mu , \name{\qone_\mu} \mid \mu < \kappa\ra$ is a Magidor iteration of Prikry type forcing, where 
For each $\mu < \kappa$ the forcing $\qone_\mu$ is nontrivial if and only if there are $x \in S'$, $\nu \in \Omega'_x$, and $i \in x \cap \nu$ 
such that $\mu = \Theta_i(\nu)$. Note that $\nu$ is the unique such cardinal and $\mu = \Theta_i(\nu)$ is not a limit of measurable cardinals, thus $\Theta_i(\nu)$ carries a unique normal measure 
$U_{\Theta_i(\nu)}$ in $V$. Furthermore, as $\pzero$ factors into $\pzero_{\nu+1} * \pzero \setminus (\nu +1)$ where
$|\pzero_{\nu+1}| < \Theta_i(\nu)$ and $\pzero \setminus (\nu +1)$ is $(2^{\Theta_i(\nu)})^+-$distributive, we get that $U_{\Theta_i(\nu)}$ 
has a unique extension $U^0_{\Theta_i(\nu)}$ in $V^0$.  Similarly, the fact $\mu$ is not a limit of measurable cardinals implies that $|\pone_\mu| < \mu$. By an argument of Levy and Solovay \cite{Levy-Solovay},
$U_{\Theta_i(\nu)}$  has a unique extension  $U^1_{\Theta_i(\nu)}$ in a $\pone_\nu$ generic extension of $V^0$.\\
{Define} $\qone_\mu = Q(U^1_{\Theta_i(\mu)})$ where $Q(U^1_{\Theta_i(\mu)})$ is the one-point Prikry forcing by $U^1_{\Theta_n(\mu)}$, and introduces a single Prikry point $d(\mu) < \mu$.
\end{definition}

\begin{definition}[$\F$,$Z_{\alpha,x,\c}(j)$, local Prikry functions]${}$\\
Let $\alpha < l(\vec{F})$, $x \in S'$, and $c :\lambda \to \omega$.
\begin{enumerate}
\item 
Let $\mathcal{F}$ denote the set of functions 
$f : \kappa \to \power(\kappa)$ in $V^0$, such that  $f(\mu) \in U^0_\mu$ for every nontrivial iteration stage $\mu < \kappa$.
Note that $|\mathcal{F}| = \kappa^+$ and $j^0_{\alpha,x}(f)(\theta_j) \in U^0_{\theta_j}$ for every $f \in \mathcal{F}$ and $j < \lambda$. 
Furthermore, as $V^0_{\theta+1} \subset M^0_{\alpha,x}$ it follows that 
$\{ j^0_{\alpha,x}(f)(\theta_j) \mid f \in \F\} \in M^0_{\alpha,x}$ and that $|\{ j^0_{\alpha,x}(f)(\theta_j) \mid f \in \F\}|^{M^0_{\alpha,x}} = |\F| = \kappa^+$
for every $j < \lambda$.  Thus $\bigcap\{j^0_{\alpha,x}(f)(\theta_j) \mid f \in \F\} \in U^0_{\theta_j}$.

\item 
For every $j < \lambda$, let 
$Z_{\alpha,x}(j) = \bigcap\{j^0_{\alpha,x}(f)(\theta_j) \mid f \in \F\}$.
\item
 For every $c: \lambda \to \omega$ let
$Z_{\alpha,x,\c}(j) = i^0_{\alpha,x,c}(Z_{\alpha,x}(j))$.
\item 
Suppose that $c : \lambda \to \omega$ is $x-$suitable.
We say that a function $\delta \in \prod_{j \in x} \theta_j^{\c(j)}$
is a \emph{local Prikry function} with respect to $\alpha,x,c$ when 
	\begin{itemize}
	 \item $\delta(j) \in Z_{\alpha,x,c}(j)$ for every $j \in x$. 
	 \item $\delta(j) = \theta_j$ for all but finitely many $j \in x$. 
	\end{itemize}
\end{enumerate}
\end{definition}

Note that for every $j < \lambda$,
$Z_{\alpha,x,\c}(j) \in U^0_{\theta^{c(j)}_j} = i^0_{c}(U^0_{\theta_j})$ is not empty. 
Furthermore, if $c(j) > 0$ then the $c-$derived iteration includes an ultrapower
by $U^0_{\theta_j}$ whose critical point is $\theta_j$, thus 
$\theta_j \in i^0_{\alpha,x,c}(Z_{\alpha,x}(j)) = Z_{\alpha,x,\c}(j)$.\\
 
Let $\Gone \subset \pone$ be a $V^0-$generic filter and denote $V^0[\Gone] = V[\Gzero * \Gone]$ by $V^1$.

\begin{definition}[$i^0_{\alpha,x,c\uhr \sigma}$, $N^0_{\alpha,x,c\uhr \sigma}$, $k^0_{\alpha,x,c\uhr \sigma}$, $(\alpha,c,\delta,\Gone)-$compatible conditions]\label{Definition - local Prikry function and compatible conditions}${}$\\
Let $\alpha < l(\vec{F})$ and $x \subset \lambda$, and suppose that $c : \lambda \to \omega$ is a $x-$suitable funciton. 
\begin{enumerate}
\item For every $\sigma \in \power_{\omega}(\lambda)$
let $i^0_{\alpha,x,c\uhr \sigma} : M^0_{\alpha,x} \to N^0_{\alpha,x,c\uhr \sigma}$ be the embedding obtain from restricting the iteration of $i^0_{\alpha,x,c}$ to the ultrapowers $\Ult^{c(j)}(*,U^0_{\theta_j})$ for $j \in \sigma$.  

\item 
It follows that $N^0_{\alpha,x,c}$ is the direct limit of $\{N^0_{\alpha,x,c\uhr \sigma} \in \power_{\omega}(\lambda)\}$ with obvious embeddings. For every $\sigma \in \power_{\omega}(\lambda)$ let 
$k^0_{\alpha,x,c\uhr \sigma} : N^0_{\alpha,x,c\uhr \sigma} \to N^0_{\alpha,x,c}$
be the resulting limit embedding. 

\item Let $\mu < \kappa$ so that $\mu$ is a nontrivial stagein $\pone$. Suppose $\tau$ is a $\pone_\mu$ name for an ordinal less than $\mu$. For every $p \in \pone$ let $p^{+(\tau,\mu)}$ be the condition obtained by replacing $p_{\mu}$ with $\tau$. Therefore $p^{+(\tau,\mu)} \force \name{d}(\can{\mu}) = \tau$. Note that $p^{+(\tau,\mu)} \geq p$ whenever $p\uhr \mu \force \tau \in p_\mu$. 

\item Suppose $\delta$ is a local Prikry function with respect
to $\alpha,x,c$.
For every $p \in \pone$ and $q \in j^0_{\alpha,x,c}(\pone)$, we say 
that $q$ is $(\alpha,c,\delta,p)-$compatible if there are $\sigma \in \power_\omega(x)$($x = \dom(\delta)$) and $q' \in i^0_{\alpha,x,c\uhr \sigma} \circ j^0_{\alpha,x}(\pone)$ so that 
\begin{enumerate}
\item $q' \geq^* \left(i^0_{\alpha,x,c\uhr \sigma} \circ j^0_{\alpha,x}(p)\right)^{+ \la (\delta(j), \theta_j^{\c(j)}) \mid j \in \sigma\ra}$.
  \item $q' \uhr \kappa = p$.
  \item $q = k^0_{\alpha,x,c\uhr \sigma}(q')$.
\end{enumerate}
We refer to $\sigma \subset \lambda$ as the \emph{ultrapower support} of $q$, and denote it by $\sigma(q)$. It is clearly unique.

\item
We say that $q \in j^0_{\alpha,x,c}(\pone)$ is  
$(\alpha,c,\delta,\Gone)-$compatible 
if it is $(\alpha,c,\delta,p)-$compatible for some $p \in \Gone$.
\end{enumerate}
\end{definition}

\begin{definition}[$F^1_{\alpha,c,\delta}$]\label{Definition - II - F_c,delta}${}$\\
Let $\alpha < l(\vec{F})$, $x \in S'$, and suppose that $c$
is a $x-$suitable function and $\delta$ is a local Prikry function with respect to $\alpha,x,c$.
Define a $(\kappa,\theta^+)-$extender $F^1_{\alpha,c,\delta}$ in $V^1$ as follows:
For every $\gamma <  \theta^+$ and $Y = (\name{Y})_{\Gone} \subset \kappa$,
$Y \in F^1_{\alpha,c,\delta}(\gamma)$
if and only if there is a $(\alpha,c,\delta,\Gone)-$compatible condition $q \in j^0_{\alpha,x,c}(\pone)$ so that $q \force \can{\gamma} \in j^0_{\alpha,x,c}(\name{Y})$.
\end{definition}

Clearly,  $F^1_{\alpha,c,\delta}$ extends $F^0_{\alpha,x,c}$.\\

The following notations will be useful in the analysis of $F^1_{\alpha,c,\delta}$. 
\begin{definition}[$\vec{\theta}_j^c$, $\vec{\theta}_\sigma^c$]
\begin{enumerate}
\item 
For every $j < \lambda$ and $c : \lambda \to \lambda$, let $\vec{\theta}_j^c = \la \theta_j^0, \dots,\theta_j^{c(j)-1}\ra$
denote the sequence of critical points of the $c(j)-$the iterated ultrapower
by $U^0_{\theta_j}$. 
\item 
Suppose that $\sigma = \{j_0,\dots,j_{m-1}\}$ is a finite subset of $\lambda$, 
let $\vec{\theta}_\sigma^c = \vec{\theta}^c_{j_0} \fr \dots \fr \vec{\theta}^c_{j_{m-1}}$. 
\end{enumerate}
\end{definition}
Note that $\vec{\theta}_\sigma^c$ enumerates the generators of $i^0_{\alpha,x,c\uhr \sigma}$.

\begin{lemma}\label{Lemma - II - each F^1_omega alpha is kappa complete}
$F^1_{\alpha,c,\delta}(\gamma)$ is a $\kappa-$complete
ultrafilter for every $\gamma < \theta^+$. 
\end{lemma}
\begin{proof}
Let $p$ be a condition in $\Gone$ and suppose that $\name{f}$ is a 
$\pone$ name for a function from $\kappa$ to some $\beta < \kappa$.
Note that $\name{f}$ belongs to $N^0_{\alpha,x,c\uhr \sigma}$ for every relevant
$\alpha,x,c,\sigma$.\\
Choose $\sigma = \{j_0,\dots,j_{m-1}\} \in \power_\omega(\lambda)$ and $\gamma' \in N^0_{\alpha,x,c\uhr \sigma}$ so that $\gamma = k^0_{\alpha,x,c\uhr \sigma}(\gamma')$. \\
Let $t = \left(i^0_{\alpha,x,c\uhr \sigma}\circ j^0_{\alpha,x}(p)\right)^{+ \la (\delta(j), \theta_j^{\c(j)}) \mid j \in \sigma\ra}$.
Since $\pone$ satisfies the Prikry condition and the direct extension order in
$\left(i^0_{\alpha,x,c\uhr \sigma}\circ j^0_{\alpha,x}(\pone)\right) \setminus \kappa$ is $\kappa^+-$closed,there are $p' \geq p$ in $\Gone$ and 
\[t' \geq^* \left(i^0_{\alpha,x,c\uhr \sigma}\circ j^0_{\alpha,x}(p')\right)^{+ \la (\delta(j), \theta_j^{\c(j)}) \mid j \in \sigma\ra} \setminus \kappa
\] so that the condition
$q' = p' \fr t'$ decides the statement $\can{\gamma'} \in i^0_{\alpha,x,c\uhr \sigma}\circ j^0_{\alpha,x}(\name{f}(\can{\nu}))$, for every $\nu < \beta$. 
Let  $\nu^*$ be the unique $\nu < \beta$ so that
 $p' \fr t' \force \can{\gamma'} \in i^0_{\alpha,x,c\uhr \sigma}\circ j^0_{\alpha,x}(\name{f}(\can{\nu}))$. 
We conclude that $q = k^0_{\alpha,x,c\uhr\sigma}(q')$ is $(\alpha,c,\delta,\Gone)-$compatible and forces that $\can{\gamma} \in j^0_{\alpha,c,x}(\name{X_{\nu^*}})$, thus $X_{\nu^*} \in F^1_{\alpha,c,\delta}(\gamma)$. 
\end{proof}
A similar argument shows that $F^1_{\alpha,c,\delta}(\kappa)$ is a normal measure on $\kappa$ in $V^1$. \\

\begin{remark}\label{remark before F completeness}
Suppose that $c : \lambda \to \omega$ is $x-$suitable 
and $\delta$ is a local Prikry function with respect to $\alpha,x,c$. 
It follows there is a finite set $\sigma^* \subset \lambda$
so that for every $j \in \lambda \setminus \sigma^*$,
$(c(j) = 0)$ if $j \not\in x$, and
$(c(j) = 1) \wedge (\delta(j) = \theta_j = \theta^0_j)$ if $j \in x$.
We can also assume that $c(j) > 1$ for every $j \in \sigma^*$, therefore
the generators of the $c-$derived iteration are 
the ordinals in $\vec{\theta}^c_{\sigma^*} \fr \vec{\theta}_{x \setminus \sigma^*}$,
where $\vec{\theta}_{x \setminus \sigma^*} = \la \theta_j \mid j \in x \setminus \sigma^*\ra$
and $\la \theta_j \mid j \in x \setminus \sigma^*\ra = \la \delta(j) \mid j \in x \setminus \sigma^*\ra$. \\
Working in $V^1$, let $\vec{d}_x : \Omega'_x \to \kappa^{<\kappa}$ denote the
local Prikry function, where for every $\nu \in \Omega'_x$, 
$\vec{d}_x(\nu) = \la d(\Theta_j(\nu)) \mid j \in x\cap \nu\ra$. 
Since $q$ is $(\alpha,c,\delta,\Gone)-$compatible we get that for every finite $\sigma \subset x \setminus \sigma^*$,
if $\sigma \subset \sigma(q)$ then
$q \force j^0_{\alpha,x,c}(\name{\vec{d}}_x)(\can{\kappa})\uhr \sigma = 
\la \theta_j \mid j \in \sigma\ra = \vec{\theta}^c_\sigma$. 
\end{remark}

\begin{lemma}\label{Lemma - II - F_omega is complete}
Suppose that $c : \lambda \to \omega$ is $x-$suitable 
and $\delta$ is a local Prikry function with respect to $\alpha,x,c$, 
then $F^1_{\alpha,c,\delta}$ is a $\kappa-$complete extender in
$V^1$.
\end{lemma}

\begin{proof}
It is sufficient to prove that $\Ult(V^1,F^1_{\alpha,c,\delta})$ is closed under
$\kappa-$sequence of ordinals below $\theta^+$. 
Let $\la \gamma_i \mid i < \kappa\ra$ be a sequence in $V^1$ of ordinals below $\theta^+$.
We use the notations and observations given in the previous remark (\ref{remark before F completeness}).\\
For every $i < \kappa$ there is a finite set $\sigma_i \subset x \setminus \sigma^*$ and a function $f_i \in M^0_{\alpha,x}$ so that
$\gamma_i = i^0_{\alpha,c,x}(f_i)(\vec{\theta}^c_{\sigma^*} \fr \vec{\theta}^c_{\sigma_i})$. \\
Since $\vec{\theta}^c_{\sigma^*}$ is a finite sequence of ordinals below $\theta^+$ we have that $\vec{\theta}^c_{\sigma^*} = j^0_{\alpha,x,c}(F^{\sigma^*})(\gamma^*)$ for some $\gamma^* < \theta^+$ and a function
$F^{\sigma^*} \in V^0$.\\ 
Let $\vec{f}$ denote $\la f_i \mid i < \kappa\ra$ and $\vec{\sigma} = \la \sigma_i \mid i < \kappa\ra$. 
The model $M^0_{\alpha,x} \cong \Ult(V^0,F^0_{\alpha,x})$ is closed under $\kappa-$sequence so $\vec{f},\vec{\sigma} \in M^0_{\alpha,x}$.
Note that $j^0_{\alpha,x,c}(\vec{\sigma})\uhr \kappa = \vec{\sigma}$ and that $i^0_{\alpha,x,c}(\vec{f}) = \la i^0_{\alpha,c,x}(f_i) \mid i < \kappa\ra$. 
Since $i^0_{\alpha,x,c}(\vec{f}) \in M^0_{\alpha,x,c}$, we may assume that
there is a function $F^{\vec{f}} \in V^0$ so that
$j^0_{\alpha,x,c}(F^{\vec{f}})(\gamma^*) = i^0_{\alpha,x,c}(\vec{f})$.\\

Let $\Omega^*_x \subset \kappa$ denote the set of $\mu < \kappa$ so that
$\max(\Omega'_x \cap \mu+1)$ exists, and denote it by $\kappa(\mu)$. 
Note that $\Omega^*_x \in F^0_{\alpha,x,c}(\gamma)$ for every  $\kappa \leq \gamma < \theta^+$, and that $j^0_{\alpha,x,c}([\mu \mapsto \kappa(\mu)])(\gamma) = \kappa$. \\
Let $\name{h}$ be a $\pone$ name for the function $h \in V^1$, so that  $h : \Omega^*_x \to \kappa^{<\kappa}$ and for $\mu \in \Omega^*_x$ we have that
\begin{itemize}
\item $h(\mu)$ is a function, $\dom(h(\mu)) = \kappa(\mu)$,
\item for every $i < \kappa(\mu)$, 
$h(\mu)(i) = \vec{d}_x(\kappa(\mu)) \uhr (\sigma_i \cap \kappa(\mu)) = 
\la d\left( \Theta_j(\kappa(\mu)) \right) \mid j \in \sigma_i \cap \kappa(\mu)\ra$. 
\end{itemize}
Note that this definition makes sense as $\sigma_i \subset x$
and $\kappa(\mu) \in \Omega'_x$. \\
We get that for every $\gamma < \theta^+$, $j^0_{\alpha,x,c}(\name{h})(\gamma)$ is a $j^0_{\alpha,x,c}(\pone)$ name of a function 
with domain $\kappa$, and for every $i < \kappa$,
\[j^0_{\alpha,x,c}(\name{h})(\gamma)(i) = 
\la j^0_{\alpha,x,c}(\name{d})(\theta^c_j) \mid j \in \sigma_i \ra.\]
If $q$ is a 
$(\alpha,c,\delta,\Gone)-$compatible condition and $\sigma_i \subset \sigma(q)$, 
then 
\[q \force j^0_{\alpha,x,c}(\name{h})(\gamma)(i) = \la 
\theta_j \mid j \in \sigma_i \ra = \vec{\theta}^c_{\sigma_i}.\]

\noindent Finally, let $\name{g}$ be the $\pone$ name of the function
$g \in V^1$ so that for every $\mu \in \Omega^*_x$,
\begin{itemize}
\item $g(\mu)$ is a function, $\dom(g(\mu)) = \kappa(\mu)$,
\item for every $i < \kappa(\mu)$, 
$g(\mu)(i) = \left(F^{\vec{f}}(\mu)_i\right)(F^{\sigma^*}(\mu) \fr h(\mu))$.
\end{itemize}

Therefore if $q$ is a 
$(\alpha,c,\delta,\Gone)-$compatible and $\sigma_i \subset \sigma(q)$, 
then 
\[q \force j^0_{\alpha,x,c}(\name{g})(\gamma^*)(i)=
\left(f_i\right)(\vec{\theta}^c_{\sigma^*} \fr \vec{\theta}^c_{\sigma_i}) = \gamma_i.\]

We conclude that in $V^1$, for every $i < \kappa$,
\[ \{ (\mu_0, \mu_1) \in \kappa^2 \mid \mu_0 = g(\mu_1) \} \in F^1_{\alpha,c,\delta}(\gamma_i) \times \F^1_{\alpha,c,\delta}(\gamma^*),\] hence $\la \gamma_i \mid i < \kappa\ra$ is represented by $[g,\gamma^*]_{F^1_{\alpha,c,\delta}} \in \Ult(V^1,F^1_{\alpha,c,\delta})$.
\end{proof}

\begin{definition}[$j^1_{\alpha,c,\delta}$, $M^1_{\alpha,c,\delta}$]${}$\\
Let $\alpha < l(\vec{F})$ and $x \in S'$, and suppose that $c : \lambda \to \omega$ is $x-$suitable and $\delta$ is a local Prikry function with respect to $\alpha,x,c$.\\
Let $j^1_{\alpha,c,\delta} : V^1 \to M^1_{\alpha,c,\delta} \cong \Ult(V^1,F^1_{\alpha,c,\delta})$ be the ultrapower embedding of $V^1$ by $F^1_{\alpha,c,\delta}$, and $G^1_{\alpha,c,\delta} = j^1_{\alpha,c,\delta}(\Gone)$. 
\end{definition}
Lemma \ref{Lemma - II - F_omega is complete} implies that $M^1_{\alpha,c,\delta}$ is closed
under $\kappa-$sequences in $V^1$. 
We also see that $G^1_{\alpha,c,\delta} \subset j^0_{\alpha,x,c}(\pone)$ is generic
over $M^0_{\alpha,x,c}$ and that $j^1_{\alpha,c,\delta}\uhr V^0 = j^0_{\alpha,x,c}$. \\

\subsubsection{The restriction $j^1_{\alpha,c,\delta} \uhr V$}
We proceed to describe the restrictions of the embeddings $j^1_{\alpha,c,\delta}$ to $V^0$ and $V$. We define an iterated ultrapower of $V^0$, generating an
embedding $\pi^0_{\alpha,c,\delta} : V^0 \to  Z^0_{\alpha,c,\delta}$ which coincides with the restriction of $j^1_{\alpha,c,\delta}\uhr V^0$. 
The technical arguments justifying $\pi^0_{\alpha,c,\delta} = j^1_{\alpha,c,\delta}\uhr ^0$ will be similar to those given in \cite{OBN - Mitchell order I} and are therefore omitted here. \\
Let $\T^0_{\alpha,c,\delta} = \la Z^0_i, \sigma^0_{i,j} \mid i<j<\zeta\ra$ be a normal iteration defined by
$Z^0_0 = V^0$
$\sigma^0_{0,1} : Z^0_0 \to Z^0_1$ is the ultrapower embedding of $V^0$ by 
$F^0_{\alpha,x,c}$ ($x = \dom(\delta)$).
Let $\vec{\mu} = \la \mu_i \mid i < \zeta\ra$ be an increasing enumeration of all the ordinals $\mu$ in 
$\sigma^0_{0,1}(\Omega^*)$, i.e., which are of the form 
$\mu = \sigma^0_{0,1}(\Theta_i)(\nu)$ where $\nu \in j^0_{\alpha,x,c}(\Omega'_y)\setminus \theta^+$ for some $y \in S'$ and $i \in \sigma^0_{0,1}(y) \cap \nu$.
Therefore $\vec{\mu}$ enumerates all the nontrivial forcing stages in the iteration 
$\sigma^0_{0,1}(\pone)\setminus (\kappa + 1)$.\\
For every $i  <\zeta$ let $U^0_{\mu_i}$ be the
unique normal measure on $\mu_i$ in $Z^0_1 = N^0_{\alpha,x,c}$. 
For every $0 < i < \zeta$, let 
$\sigma^0_{i,i+1} : Z^0_i \to Z^0_{i+1} \cong \Ult(Z^0_i,\sigma^0_{0,i}(U^0_{\mu_i}))$.
For every limit $i^* \leq \zeta$, $Z^0_{i^*}$ is the direct limit of the iteration 
$\T_{F,\c}\uhr i^* = \la Z^0_i, \sigma^0_{i,j} \mid i<j<i^*\ra$.
Then $\pi^0_{\alpha,c,\delta} = \sigma^0_{0,\zeta}$ and $Z^0_{\alpha,c,\delta} = Z^0_{\zeta}$. 

\begin{corollary}\label{Corollary - II - Genericity of G1 F c delta}
The restriction of $j^1_{\alpha,c,\delta} : V^1 \to M^1_{\alpha,c,\delta}$ to $V^0$
is $\pi^0_{\alpha,c,\delta} : V^0 \to Z^0_{\alpha,c,\delta}$.
\end{corollary}
Let $\Gone_{\alpha,\c,\delta} = j^1_{\alpha,c,\delta}(\Gone)$.
Then $\Gone_{\alpha,\c,\delta} \subset \pi^0_{\alpha,c,\delta}(\pone)$ is
generic over $Z^0_{\alpha,c,\delta}$. \\

It is clear that the description of $j^1_{\alpha,c,\delta}$ to $V$ results
from the restriction of the $V^0$ iteration $\T^0_{\alpha,c,\delta} = \la Z^0_i, \sigma^0_{i,j} \mid i<j<\zeta\ra$ to $V$, where the embedding 
$\sigma^0_{0,1} = j^0_{\alpha,x,c}$ is replaced with
$\sigma_{0,1} = j_{\alpha,c}$,
and for every $i  < \zeta$, $\sigma^0_{i,i+1} : Z^0_i \to Z^0_{i+1} \cong \Ult(Z_i,\sigma^0_{0,i}(U^0_{\mu_i}))$
is replaced with $\sigma_{i,i+1} : Z_i \to Z_{i+1} \cong \Ult(Z_i,\sigma_{0,i}(U_{\mu_i}))$. We denote the limit model by $Z_{\alpha,c}$. Thus $Z_{\alpha,c} =  \K(Z^0_{\alpha,c,\delta}) = \K(M^1_{\alpha,c,\delta})$. 

\subsubsection{{Further results concerning $\pone$}}
We conclude this subsection with a few results concerning the poset $\pone$ and extenders $F^1_{\alpha,c,\delta}$, which will be used in the study of the normal measures on $\kappa$ in the next generic extension $V^2$.

\begin{lemma}\label{Lemma - II - Stronger Definition for F1 c delta alpha}
 For every $\gamma < \theta^+$ and $X = \name{X}_{\Gone} \in F^1_{\alpha,c,\delta}(\gamma)$
 there is $t \in \Gone$ and a $(\alpha,c,\delta,p)$ compatible condition $q \in j^0_{\alpha,x,c}(\pone)$ so that
 \begin{itemize}
 \item $t \force \can{\gamma} \in j^0_{\alpha,x,c}(\name{X})$,
 \item  $t \setminus \theta^+ = j^0_{\alpha,x,c}(p) \setminus \theta^+$
\end{itemize}
\end{lemma}

\begin{proof}
Suppose that $q \in j^0_{\alpha,x,c}(\pone)$ be a $(\alpha,c,\delta,p)-$compatible condition so that $q \force \can{\gamma} \in j^0_{\alpha,x,c}(\name{X})$. 
Let $\sigma = \sigma(q)$ be the support of $q$ and $q' \in N^0_{\alpha,x,c\uhr\sigma}(\pone)$ so that $q' \geq^* \left(i^0_{\alpha,x,c\uhr \sigma} \circ j^0_{\alpha,x}(p)\right)^{+ \la (\delta(j), \theta_j^{\c(j)}) \mid j \in \sigma\ra}$
and $q = k^0_{\alpha,x,c\uhr\sigma}(q')$. Let us denote 
$i^0_{\alpha,x,c\uhr \sigma} \circ j^0_{\alpha,x}$ by $j^0_{\alpha,x,c\uhr\sigma}$.\\
Note that $\theta^+$ is a fixed point of $k^0_{\alpha,x,c\uhr\sigma}$.
Let $r = q' \setminus \theta^+$. $r$ is a $j^0_{\alpha,x,c\uhr\sigma}(\pone)\uhr \theta^+$ name for a condition in $j^0_{\alpha,x,c\uhr\sigma}(\pone)\setminus\theta^+$, and $q' \uhr \theta^+ \force r' \geq^* j^0_{\alpha,x,c\uhr\sigma}(p)\setminus \theta^+$. \\
Take $R \kappa \to \kappa$ and $\gamma < \theta^+$ so that
$r = j^0_{\alpha,x,c\uhr\sigma}(R)(\gamma)$. We may assume that for every $\nu \in \Omega'$ and $\mu \in [\nu,\Theta(\nu)^+)$, $R(\mu) \geq^* p\setminus \Theta(\nu)^+$. 
For every $\mu \in \Omega^* \setminus \supp(p)$ and $\nu  = \max(\Omega' \cap \mu)$ let $p^*_\mu$ be a $\pone_\mu$ name
of $\bigcap\{ R(\mu')_\mu \mid \mu' \in \Omega^* \cap \nu\}$. 
The fact $|\Omega^* \cap \nu| \leq \nu < \mu$ implies that
$p^*_\mu$ is a $\pone_\mu$ name for a set in $Q(U^1_\mu)$.
We conclude that $p^* \geq^* p$, and that for every $\nu \in \Omega'$ and $\mu' \in [\nu,\Theta(\nu)^{+})$, $p^*\uhr \Theta(\nu)^{+} \force p^*\setminus \Theta(\nu)^{+} \geq^* R(\mu')$. 
Let $t' = q' \uhr \theta^+ \fr \left(j^0_{\alpha,x,c\uhr\sigma}(p^*)\setminus \theta^+\right)$ and $t = k^0_{\alpha,x,c\uhr \sigma}(t')$.
Then $t \in j^0_{\alpha,x,c}(\pone)$ is $(\alpha,c,\delta,p^*)-$compatible extension of $q$. The Lemma follows from a standard density argument.
\end{proof}

 \begin{lemma}\label{Lemma - II - cell choosing ordinals}
Suppose that $\po = \po_0 * Q_\mu * \po_1$ is an iteration of Prikry forcings, so that  
\begin{enumerate}
 \item $|\po_0| < \mu$,
 \item $Q_\mu = Q(U_\mu)$ as a one point Prikry forcing, 
 \item $\leq^*_{\po_1}$ is $\mu^+-$closed. 
 \end{enumerate}
Let $p \in \po$ such that $p_{Q_\mu} \in U_\mu$, and let
$\name{\eta}$ be a $\po-$name of an ordinal so that $p \force \name{\eta} < \can{\mu}$.
There are $p^* \geq^* p$ and a function $\phi : \mu \to [\mu]^{|\po_0|}$ so that
\begin{enumerate}
 \item $p^*\uhr(\po_0 * Q_\mu)  =  p\uhr(\po_0 * Q_\mu)$ and
 \item $p^* \force \name{\eta} \in \can{\phi}(\name{d}(\can{\mu}))$,
\end{enumerate}
where $\name{d(\mu)}$ is the $Q_\mu-$name for the generic Prikry point. 
\end{lemma}

\begin{proof}
 Since $\leq^*_{\po_1}$ is $\mu^+-$closed and $p \force \name{\eta} < \mu$, there is an is a $p^* \geq^* p$ with $p^*\uhr(\po_0 * Q_\mu)  =  p\uhr(\po_0 * Q_\mu)$, so that $p^*$ reduces $\name{\eta}$ to a  $\po_0 * Q_\mu$ name, 
 namely,  $p^* \force \name{\eta} = \tau$ where
$\tau$ is a $\po_0 * Q_\mu$ name for an ordinal below $\mu$.\\
For every ordinal $\nu$ let $A_\nu \subset \po_0$
be a maximal antichain of conditions which decide $\can{\nu} \in p_{Q_\mu}$.
Note that if $p_0 \in A_\nu$ forces $\can{\nu} \in p_{Q_\mu}$, 
the forcing $\po_0 * Q_\mu$ above $(p_0 * p_{Q_\mu})^{+(\nu,\mu)}$ is equivalent to the forcing $\po_0$ above $p_0$.
For each $p_0 \in A_\delta$ which forces $\can{\nu} \in p_{Q_\mu}$,
let $A_{\nu,p_0}$ be a maximal antichain above $p_0$, 
which consists of conditions which decides the value of $\sigma$ as an ordinal below $\mu$. Let $\phi(\nu) \subset \mu$ be the collection of all possible values of $\sigma$ forced by conditions $A_{\nu,p_0}$ for suitable $p_0 \in A_\nu$. 
It is clear that $|\phi(\nu)| \leq |\po_0|$ and that
$p\uhr (\po_0 * Q_\mu) \force \text{ if }\name{d}(\can{\mu}) = \can{\nu} \text{ then } \tau \in \can{\phi(\nu)}$.
\end{proof}

By applying the previous Lemma consecutively $\omega-$times, we conclude the following result:
\begin{corollary}\label{Corollary - II - guessing omega sequences}
Let $\vec{\mu} = \la \mu_n \mid n < \omega\ra$ be a sequence
of non-trivial $\pone$ forcing stages, and
$\vec{v} = \la \name{\eta_n} \mid n < \omega\ra$ be a $\pone-$name of an ordinal sequence. Suppose $p \in \pone$ forces that $\vec{v} \in \prod_{n<\omega}\can{\Theta_{\mu_n}}$, then there is $p^* \geq^* p$, and a sequence of functions
$\vec{\psi} = \la \psi_n \mid n <\omega\ra$ in $V^0$ so that
\begin{enumerate}
 \item  $\psi_n : \mu_n \to [\mu_n]^{|\pone_{\mu_n}|}$ for every $n < \omega$.
 \item  $p^* \force \forall n<\omega.\thinspace  \name{\eta_n} \in \can{\psi_n}\left(\name{d}(\can{\mu_n})\right).$
\end{enumerate}
\end{corollary}


\subsection{The Poset $\ptwo$}\label{II subsection - P2}
\begin{definition}[$\ptwo$]${}$\\
$\ptwo = \la \ptwo_\nu, \qtwo_\nu \mid \nu \leq \kappa\ra$ is a collapsing and coding iteration with a Friedman-Magidor (nonstationary) support.
The nontrivial stages of the iteration are at $\nu \in \Omega' \cup\{\kappa\}$,
and $\qtwo_\nu = \Coll(\nu^+, \Theta(\nu)^+) * \Code(\name{g_\nu})$ where
\begin{itemize}
\item 
$\Coll(\nu^+,\theta(\nu)^+)$ is a Levy collapsing, which introduces a surjection
$g_\nu : \nu^+ \to \Theta(\nu)^+$.
\item $\Code(\name{g_\nu})$ codes $g_\nu$ and itself using a closed unbounded set in $\nu^+$, as in Definition \ref{Definition - Collapse and Code}.
\end{itemize}
\end{definition}

Let $\Gtwo \subset \ptwo$ be a generic filter of the forcing $\ptwo$ over $V^1$.
We denote $V^1[\Gtwo] = V[\Gone*\Gtwo]$ by $V^2$. \\
The arguments in Section \ref{Section - II - First Overlapping} (addressing the extension of the embeddings $j_{F,n}$ in $V$ to a $V^{\po}$ embeddings) guarantee that the following can be defined:
\begin{definition}[$U^2_{\alpha,c,\delta}$,$j^2_{\alpha,c,\delta}$, $M^2_{\alpha,c,\delta}$]\label{Corollary - sum U2}${}$\\
Every $j^1_{\alpha,c,\delta} : V^1 \to M^1_{\alpha,\c,\delta}$ has a unique extension to a $V^1[\Gtwo]$ embedding, denoted by 
$j^2_{\alpha,\c,\delta} : V^1[\Gtwo] \to M^2_{\alpha,c,\delta}$, so that
\begin{enumerate}
\item $M^2_{\alpha,c,\delta} = M^1_{\alpha,\c,\delta}[\Gtwo_{\alpha,\c,\delta}]$  where $\Gtwo_{\alpha,\c,\delta}  = j^2_{\alpha,\c,\delta}(\Gtwo) \subset j^1_{\alpha,\c,\delta}(\ptwo)$ is generic over $M^1_{\alpha,\c,\delta}$.
\item The set $\{p \fr (j^1_{\alpha,c,\delta}(p) \setminus \kappa+1) \mid p \in \Gtwo\}$ meets every dense open set in $j^1_{\alpha,x,\delta}(\ptwo)$ and generates $\Gtwo_{\alpha,\c,\delta}$.
\item The set $U^2_{\alpha,\c,\delta} = \{X \subset \kappa \mid \kappa \in j^2_{\alpha,\c,\delta}(X)\}$ is a normal measure on $\kappa$ in $V^1[\Gtwo]$.
\item $j^2_{\alpha,\c,\delta} : V^1[\Gtwo] \to M^2_{\alpha,c,\delta}$ coincides with the ultrapower embedding and model of $V^1[\Gtwo]$ by $U^2_{\alpha,c,\delta}$.
\end{enumerate}
\end{definition}

\begin{theorem}\label{Proposition - II - W = U c delta -  the omega case}
Suppose that $W$ is a normal measure on $\kappa$ in $V^2$. Then $W = U^2_{\alpha,c,\delta}$ for some  $\alpha < l(\vec{F})$, $x \in S'$, a function  $c : \lambda \to \omega$ which is $x-$suitable, and a local Prikry function $\delta$ with respect to $\alpha,x,c$.
\end{theorem}

Let $j_W : V^2 \to M_W \cong \Ult(V^2,W)$.
$j_W\uhr V : V \to M$ results from a normal iteration $\pi_{0,b}^\T$ generated by an iteration tree $\T$ of $V$ and a branch $b$.
Moreover $M_W = M[\Gzero*\Gone_W*\Gtwo_W]$ where $\Gzero_W*\Gone_W*\Gtwo_W = j_W(\Gzero*\Gone*\Gtwo) \subset j(\pzero*\pone*\ptwo)$ is generic over $M$.\\

The description of a collapsing and coding generic extension $V^{\po}$ in Section \ref{Section - II - First Overlapping} applies to $\ptwo$. We get that in $V^2$, 
there is a sequence of functions  $\la \psi_\delta \mid \delta < \theta^+ \ra \subset {}^\kappa \kappa$ so that for every $\delta_0\neq \delta_1$,
$\{ \nu < \Omega' \mid \psi_{\delta_0}(\nu)  \neq \psi_{\delta_1}(\nu)\}$ is bounded in $\kappa$. Since $\Omega' \in W$ we get $j_W(\kappa) \geq \theta^+$. \\
The forcing $\pzero*\pone*\ptwo$ does not add cofinal $\omega-$sequences to inaccessible cardinals in $V$. Therefore $\Gzero_W*\Gone_W*\Gtwo_W$ do not add
such sequences to inaccessible cardinals in $M$. Since $M_W$ is closed under $\omega-$sequences in $V^2$, it follows that there are no ordinals $\mu$ of countable cofinality in $V^2$, which are inaccessible in $M$. 
Therefore, the iteration $T$ cannot use the same measure/extender more than finitely  many times. Since $\cp(j) = \kappa$ and $j(\kappa) \geq \theta^+$, we conclude 
that $j = k \circ j_{F'}$, where
$F'$ is an extender on $\kappa$ in $V$ with $\nu(F') = \theta^+$ and
$\cp(k) > \theta^{+}$.\\
Since $\{U_{\theta_j} \mid j < \lambda\}$ are the only normal measures which overlaps
extenders on $\kappa$ in $V$, and each can be applied finitely many times, 
we conclude that there are $\alpha < l(\vec{F})$ and $c : \lambda \to \omega$
such that $F' = i_c(F_\alpha) = F_{\alpha,c}$.\\
Let $s^{G^0_W}_{j(\kappa)} : j(\kappa) \to j(\lambda)$ be the $G^0_W-$induced generic Sacks function. Let $ i = s^{G^0_W}_{j(\kappa)}(\kappa) < \min(\kappa,\lambda)$,
and let $x \in S'$ so that $x = x_i$. It follows that $\Omega'_x \in W$, and that
the the $j_W(\pone)$ non-trivial forcing stages between $\kappa$ and $\theta^+$, are
at the points $\theta^{c(j)}_j$, $j \in x$. 
Let $\delta \in \prod_{j \in x}\theta^{c(j)}_j$ be the function defined by 
$\delta(j) = j_W(d)(\theta^{c(j)}_j)$. 

The fact that $\Omega'_x \in W$ implies that extension of
$j = j_W \uhr V$ to an embedding of $V^0$ is of the form Let $j^0 = k^0 \circ j^0_{\alpha,x,c}$, where $j^0_{\alpha,x,c}$ is the ultrapower embedding of $V^0$
by $F^0_{\alpha,x,c} = i^0_x(F_{\alpha,x})$.\\
Recall that $j^0_{\alpha,c,x} = i^0_{\alpha,x,c} \circ j^0_{\alpha,x}$ where
$j^0_{\alpha,x}$ is the ultrapower embedding of $V^0$ by $F^0_{\alpha,x}$ 
and $i^0_{\alpha,x,c}$ is the $c-$derived embedding of $M^0_{\alpha,x} \cong \Ult(V,F^0_{\alpha,x})$. 

\begin{claim}\label{Claim - W = U2, c,delta are x suitable}
$c$ is a $x-$suitable function and $\delta$ is a local Prikry function with respect
to $\alpha,x,c$.
\end{claim}

We prove the claim for the more difficult case when $x$ is infinite. The proof for the finite case $x$ is simpler. We separate the claim into a series of subclaims:

\begin{subclaim}
$\c(j) \geq 1$ for all but finitely many  $j \in x$.
\end{subclaim}
Suppose otherwise, and fix a countable set $y \subset x$ so that 
$c(j) = 0$ for all $j \in y$, i.e., $\theta_j^{c(j)} = \theta_j$. 
We have that $i^0_{\alpha,x,c}(\theta_j) = \theta_j$ for every $j \in y$, so $i^0_{\alpha,x,c}$ maps each $\theta_j$ cofinaly to itself. In particular, there is $\gamma_j < \theta_j$ so that 
$i^0_{\alpha,x,c}(\gamma_j) > \delta(j)$. 
Let $\vec{\gamma} = \la \gamma_j \mid j \in y\ra$. $\vec{\gamma} \in V[\Gone*\Gtwo]$.
Since $\Gtwo$ does not add new $\omega-$sequences and $\pzero*\pone$ satisfies $\kappa^{++}.c.c$, it follows that in $V^0$ there is a sequence $\vec{\gamma}^* = \la \gamma_n^* \mid j \in y \ra \in \prod_{j \in y}\theta_j^{c(j)}$
so that $\gamma_j \geq \delta(j)$ for every $j \in y$. 
In particular $i^0_{\alpha,x,c}(\vec{\gamma^*}) \in N^0_{\alpha,x,c} \cap V^0_{\theta+1} = M^0_{\alpha,x,c} \cap V^0_{\theta+1} = M[\Gzero_W] \cap V^0_{\theta+1}$. We therefore define in $M[\Gzero_W]$ a set $D = \{p \in j_W(\pone) \mid i^0_{\alpha,x,c}(\gamma^*_j) \cap p_{\theta^{c(j)}_j} = \emptyset \text{ for some } j \in y\}$. 
$D$ is dense in $j_W(\pone)$ since $y$ is infinite, but it is clear we cannot have $\Gone_W \cap D \neq\emptyset$.

\begin{subclaim}
$\delta(j) \leq \theta_j^{c(j)-1}$ for all but finitely many  $j \in x$.
\end{subclaim}
Suppose otherwise, then there is a countable set $y \subset x$ 
so that $\c(j) \geq 1$ and $\delta(j) > \theta_j^{c(j)-1}$  for all $j \in y$. 
For every $j \in y$ let $\sigma_j \in \power_\omega(\lambda)$ and a function $f_j$
so that $\delta(j) = i^0_{\alpha,x,c}(f_j)(\vec{\theta}^c_{\sigma_j})$.
We may assume that $j \in \sigma_j$ is the maximal ordinal in $\sigma_j$.
We separate the last generator $\theta_j^{c(j)-1}$ from $\vec{\theta}^c_{\sigma_j}$ an write $\vec{\theta}^c_{\sigma_j} = \vec{\theta}^* \fr \la \theta_j^{c(n)-1} \ra$.
If follows that $\delta(j) = i_{\c}(f_j)(\vec{\theta}^* \fr \theta_j^{c(j)-1})$ where
$f_j \in V^0$ is a function, $f_j : \prod_{j' \in y}[\theta_{j'}]^{c(j')} \to \theta_j$. 
Let $k = \sum_{j' \in \sigma_j \setminus\{j\}}c(j') \thinspace + (c(j)-1)$, $k < \omega$, and define $g_j : \theta_j \to \theta_j$ in $V^0$ by
\[ g_j(\mu) = \sup(\{f_j(\vec{\nu}^* \fr \la\mu \ra \mid \vec{\nu}^* \in \mu^{k} \})   \]
It is clear that $\delta(j) \leq i^0_{\alpha,x,c}(g_j)(\theta_j^{\c(j)-1}) < \theta_j^{\c(j)}$. 
Let $E_j \subset \theta_j$ be the set of closure points of $g_j$, then $E_j$ is closed unbounded in $\theta_j$ and $\delta(j) \not\in i_{\c}(E_j)$ as we assumed $\theta_j^{\c(j)-1} < \delta(j)$.\\
We point out that although $E_j \in V^0$ for each $j \in y$, the sequence $\vec{E} = \la E_j \mid j \in y\ra$ may not be in $V^0$. Instead, note that $\vec{E} \in V^1$ as $\ptwo$ does not introduce new $\omega$ sequence, and that $\pzero*\pone$ satisfies $\kappa^{++}.c.c$.
It follows that there is a sequence $\vec{E}^* = \la E_j^* \mid j  \in y\ra \in V^0$ so that $E_j^* \subset E_j$ for each $j \in y$. 
We get that $i^0_{\alpha,x,c}(\vec{E}^*) = \la i^0_{\alpha,x,c}(E_j^*) \mid j \in y\ra$ belongs to $M[\Gzero_W]$, 
and we can therefore define in $M[\Gzero_W]$ the set $D = \{p \in j_W(\pone) \mid  p_{\theta_j} \subset i^0_{\alpha,x,c}(E_j^*)  \text{ for some } j \in y\}$. $D$ is dense in $j_W(\pone)$ since $y$ is infinite, but we cannot have that $D \cap \Gone_W \neq \emptyset$.

\begin{subclaim}
$\delta_j \geq \theta_j^{c(j)-1}$ for all but finitely many  $j \in x$.
\end{subclaim}
Suppose otherwise and let $y \subset x$ be a countable set so that
$\c(j) \geq 1$ and $\delta(j) < \theta_{j}^{\c(j)-1}$ for every $j \in y$. 
Since $M_W \cong \Ult(V^2,W)$ is closed under $\kappa$ sequences, it follows that the 
$\la \theta_j^{\c(j)-1} \mid j \in y\ra$ belongs to $M_W = M[\Gzero*\Gone_W* \Gzero_W]$. Furthermore, $\la \theta_j^{\c(j)-1} \mid j \in y\ra \in M[\Gzero*\Gone]$
because $\ptwo$ does not add new $\omega-$sequence.
Let $\vec{v}$ be a $j_W(\pzero*\pone)$ name for this sequence. By Corollary 
\ref{Corollary - II - guessing omega sequences}
there is a sequence of functions $\vec{\phi} = \la \phi_j \mid j \in y$ in $M[\Gzero_W]$so that for every $j \in y$, $\dom(\phi_j) = \theta_j^{\c(j)}$,  $|\phi_j(\delta(j)))|  < \delta(j)$, and $\theta_j \in \phi_j(\delta(j))$.\\
As an element of $M[\Gzero_W] \cap V^0_{\theta+1} = N^0_{\alpha,x,c} \cap V^0_{\theta+1}$, we see that $\vec{\phi} = i^0_{\alpha,x,c}(F)(\vec{\theta}^c_\sigma)$ for some finite $\sigma \subset \lambda$. 
Fix an ordinal $j \in y \setminus \sigma$. 
Let $c' : \lambda \to \omega$ defined
by $c'(j) = c(j) -1$ and $c'(j') = c(j')$ for every $j' \neq j$. 
We can factor $i^0_{\alpha,x,c} : V^0 \to N^0_{\alpha,x,c}$ into
$i^0_{\alpha,x,c} = k' \circ i^0_{\alpha,x,c'}$ where $k' : N^0_{\alpha,x,c'} \to N^0_{\alpha,x,c} \cong \Ult(N_{c'}, i^0_{\alpha,x,c'}(U_{\theta_j}))$ is the ultrapower embedding of $N^0_{\alpha,x,c'}$ by $i^0_{\alpha,x,c'}(U_{\theta_j})$.
So $\cp(k') = \theta_j^{\c(j)-1}$ and $k'(\theta_j^{\c(j)-1}) = \theta_j^{\c(j)}$.\\
Let $\vec{\phi'} = i^0_{\alpha,x,c'}(F)(\vec{\theta}^c_\sigma)$. It is clear
that $\vec{\phi} = k'(\vec{\phi'})$. 
$\vec{\phi'} = \la \phi'_j \mid j < y\ra$ where
$\phi_j = k'(\phi'_j)$ for every $j \in y$. 
Moreover, since we assumed that $\delta(j) < \theta_j^{\c(j)-1} = \cp(k')$,
we get that $\phi_j(\delta(j)) = k'(\phi'_j(\delta(j)))$.\\
We conclude that $\phi'_j(\delta(j)) \subset \theta_j^{\c(j)-1}$ and  $|\phi'_j(\delta(j))| < \delta(j)$, but this implies $\theta_j^{\c(j)-1} \not\in k'(\phi'_j(\delta(j))) = \phi_j(\delta(j))$, contradicting the above.
 
\begin{subclaim}
$\c(j) =1$ for all but finitely many $j \in x$. 
\end{subclaim}
Otherwise, there would be a countable set $y \subset x$ so that $c(j) >1$ and $\delta(j) = \theta_j^{\c(j)-1}$ for all $j \in y$. 
Let $\vec{\phi} = \la \phi_j \mid j \in y \ra$ be a sequence in $M$, so that $\theta_j^{0} \in \phi_j(\theta_j^{\c(j)-1})$, where
$\phi_j(\theta_j^{\c(j)-1}) \in [\theta_j^{\c(j)}]^{\leq \theta_{j-1}^{c{j-1}}}$.
Let $\sigma \subset \lambda$ be a finite set so that  
$\vec{\phi} = N^0_{\alpha,x,c\uhr \sigma}$, and fix an ordinal $j \in y \setminus \sigma$.
We know there exists a function $f_j : \theta_j \to [\theta_j]^{\leq \theta_{j-1}}$ so that $\phi_j = i_{\c}(f_j)(\vec{\theta}^c_\sigma)$, and we can write
\begin{equation}\label{eq - II - subcase d, c(n) = 1}
\theta_j^{0} \in \left(i^0_{\alpha,x,c\uhr \sigma}(f_j)(\vec{\theta}^c_\sigma)\right)(\theta_j^{c(j)-1})
\end{equation}
Let $i' : V^0 \to M' \cong \Ult(V^0,U^0_{\theta_j})$ be the ultrapower embedding of $V$ by $U^0_{\theta_j}$. 
Note the $f_j \in M'$, so we can see \ref{eq - II - subcase d, c(n) = 1} as a statement of $M'$, where
$\theta_j^1$ is the $j-$th measurable cardinal above $\kappa$, $U^0_{\theta_j^1} = i'(U^0_{\theta_j})$ is the unique normal measure on $\theta_j^1$. Also from the  perspective of $M'$, $\theta_j^{c(j)-1} \geq \theta_j^1$ is the image of $\theta_j^1$ 
under the $c(j)-2 \geq 0$ iterated ultrapower embedding by $U^0_{\theta_j^1}$. 
Back in $V$ we get that 
\[\{\mu < \theta_j \mid \mu \in \left(i^0_{\alpha,x,c\uhr \sigma}(f_j)(\vec{\theta}^c_\sigma)\right)(\theta_j^{c(j)-2})\} \in U^0_{\theta_j}.\]
However, this is impossible as $\left(i^0_{\alpha,x,c\uhr \sigma}(f_j)(\vec{\theta}^c_\sigma)\right)(\theta_j^{c(j)-2})$ has cardinality at most $\theta_{j-1}^{c(j-1)} < \theta_j$. \\
${}$\\

\begin{subclaim}
$c(j) = 0$ for all but finitely many $j \in \lambda \setminus x$.
\end{subclaim}
Suppose otherwise. Let $y \subset \lambda \setminus x$ be a countable set,
so that $c(j) > 0$ for all $j \in y$. Let $\vec{\theta}\uhr y = \la \theta_j \mid j \in y\ra$. $\vec{\theta}\uhr y$ belongs to $M[\Gzero_W * \Gone_W]$ because
$M_W = M[\Gzero_W * \Gone_W * \Gtwo_W]$ is closed under $\omega$ sequences in $V^2$,
and $\Gtwo_W$ does not add new $\omega$ sequences of ordinals.\\

Let us assume that $x\setminus j \neq\emptyset$ for every $j \in y$.
{The case $x\setminus j  =\emptyset$ is treated similarly.}
For every $j \in y$ let $j^* = \min(x \setminus j)$.
We may assume that $c(j^*) = 1$ and that $\delta(j^*) = \theta_{j^*}$.\\

By Corollary \ref{Corollary - II - guessing omega sequences} there is a sequence of functions
$\vec{\psi} = \la \psi_j \mid j \in y\ra$ in $M[G^0_W]$, 
with $\psi_j : \theta_{j^*}^{c(j^*)} \to [\theta_{j^*}^{c(j^*)}]^{|\pone\uhr \theta^0_{j^*}|}$ so that $\theta_{j} \in \psi_j(\delta(j^*)) = \psi_j(\theta_{j^*})$ for every $j \in y$. 
Since the non-trivial $j_W(\pone)$ iteration stages between $\kappa$
and $\theta^+$ are $\theta^{c(j^*)}_{j^*}$, $j^* \in x$, it follows that
$|\pone\uhr \theta^0_{j^*}| < \theta_{j}$ for every $j \in y$. 

Since $\vec{\psi} \in M[G^0_W] \cap V^0_{\theta+1} = N^0_{\alpha,c,x} \cap V^0_{\theta+1}$, there is a finite set $\sigma \subset \lambda$ and $\vec{\psi'} \in N^0{\alpha,x,c\uhr \sigma}$ so that $\vec{\psi} = k^0_{\alpha,x,c\uhr \sigma}(\vec{\psi'})$.
Let $\vec{\psi'} = \la \psi'_j \mid j \in y\ra$.
Pick some $j \in y \setminus \sigma$. We may assume that $j^* \in \sigma$,
thus $k^0_{\alpha,x,\sigma}(\theta_{j^*}) = \theta_{j^*}$. 
It follows that $\theta_j \in k^0_{\alpha,x,c\uhr\sigma}\left(\psi'_j(\theta_{j^*})\right)$.
This is impossible as $\theta_j$ is a critical point of the embedding 
$k^0_{\alpha,x,c\uhr\sigma}$ and
$|\psi'_j(\theta_{j^*})| < \theta_{j}$.\\
\qed[Claim \ref{Claim - W = U2, c,delta are x suitable}]\\
${}$\\

Since $c$ is $x-$suitable and $\delta$ is a local Prikry function with respect to 
$\alpha,x,c$, it follows that the normal measure $U^2_{\alpha,c,\delta}$ exists in $V^2$.

\begin{claim}\label{claim - q compatible with U2 alpha c delta is in G1W}
If $p \in \Gone$ and $q \in j^0_{\alpha,x,c}(\pone)$ is a $(\alpha,\c,\delta,p)-$compatible condition so that $q \setminus \theta^+ = j^0_{\alpha,x,c}(p) \setminus \theta^+$, then $k^0(q)\in \Gone_W$. 
\end{claim}
For every $p \in \Gone$ we have that $j_W(p) \in \Gone_W$. 
Note that if $q \in j^0_{\alpha,x,c}(\pone)$ is $(\alpha,\c,\delta,p)$ compatible 
then $q \uhr \kappa \in \Gone = \Gone_W \uhr \kappa$. 
We have that for every $j \in x$, $q$ has an extension in $j^0_{\alpha,x,c}(\pone)$ which forces that $\left(j^0_{\alpha,x,c}(\pone)(\name{d})\right)(\theta^{c(j)}_j)
= \can{\delta(j)}$ and by the definition of local Prikry function $\delta$, 
we have $\delta(j) = j_W(d)(\theta^{c(j)}_j)$. 
It follows that $q \uhr \theta^+ \in G^1_W\uhr \theta^+$, and 
as $\cp(k^0) > \theta^+$, we get that
$k^0(q) \uhr \theta^+ \in G^1_W\uhr \theta^+$.
Finally, the assumption that $q \setminus \theta^+ = j^0_{\alpha,x,c}(p) \setminus \theta^+$ implies that $k^0(q)\setminus \theta^+ = j^0(p)\setminus \theta^+ = j_W(p)\setminus \theta^+$. Hence $k^0(q) \in \Gone_W$. 

\qed[Claim \ref{claim - q compatible with U2 alpha c delta is in G1W}]\\
${}$\\

Let us show that $U_{\alpha,\c,\delta} \subset W$. 
Suppose that  $X  = \name{X} \in U^2_{\alpha,\c,\delta}$.
According to Definition \ref{Corollary - sum U2}
 there is a condition $p^2 \in \Gtwo$ so that 
 \begin{equation}\label{equation 1 - W =U alpha,c,delta}
 p^2 \fr (j^1_{\alpha,c,\delta}(p^2)\setminus \kappa+1) \force \can{\kappa} \in j^1_{\alpha,c,\delta}(\name{X}).
 \end{equation}
 The definition of $j^1_{\alpha,c,\delta}$ implies there are $p^1 \in \Gone$ and a
$(\alpha,c,\delta,p^1)-$compatible condition $q^1 \in j^0_{\alpha,x,c}(\pone)$ so
that $q^1$ forces a $j^0_{\alpha,x,c}(\pone)-$statement equivalent to \ref{equation 1 - W =U alpha,c,delta} (where $j^1_{\alpha,c,\delta}$ is replaced with $j^0_{\alpha,x,c}$ and
$p^2,\can{\kappa},\name{X}$ are replaced with their $\pone-$names). Furthermore
by Lemma \ref{Lemma - II - each F^1_omega alpha is kappa complete}
we may assume that
$q^1 \setminus \theta^+ = j^0_{\alpha,x,c}(p^1)\setminus \theta^+$.\\
It follows that $k^0(q^1)$ forces a $j_W(\po^1)-$statement equivalent to
 \begin{equation}\label{equation 2 - W =U alpha,c,delta}
 p^2 \fr (j_W(p^2)\setminus \kappa+1) \force \can{\kappa} \in j_W(\name{X}),
 \end{equation}
and Claim \ref{claim - q compatible with U2 alpha c delta is in G1W} (showing $k^0(q^1) \in \Gone_W$) guarantees \ref{equation 2 - W =U alpha,c,delta} holds. 
Finally, the coding posets in $\ptwo$ and the fact $j``\Gtwo \subset \Gtwo_W$ guarantee that $p^2 \fr (j(p^2)\setminus \kappa+1) \in \Gtwo_W$, hence $X \in W$.
\qed[Theorem \ref{Proposition - II - W = U c delta -  the omega case}]\\

\begin{proposition}\label{Proposition - II - MO in V2}
For every $U^2_{\alpha,\c,\delta}$, $U^2_{\alpha',{\c}',\delta'}$ in $V^2 = V[\Gzero*\Gone*\Gtwo]$, 
$U^2_{\alpha',{\c}',\delta'} \mo U^2_{\alpha,\c,\delta}$ if and only if 
$\alpha' < \alpha$ and $\c' \geq {\c}$.
\end{proposition}
\begin{proof}
Suppose that $\alpha' < \alpha$  and $\c' \leq \c$, and let us verify that $U^2_{\alpha',c',\delta'} \in M^2_{\alpha,c,\delta} \cong \Ult(V^2,U^2_{\alpha,c,\delta})$. \\
Let $x = \dom(\delta) \subset \lambda$. Then $x \in S'$ and $c$ is a $x-$suitable unction. We described the embedding $j = j^2_{\alpha,c\delta} \uhr V = j^1_{\alpha,c,\delta}\uhr V$ as an iterated ultrapower of $V$, $j : V \to Z_{\alpha,c}$. We know that
\begin{enumerate}
\item  $j = k  \circ j_{\alpha,c}$ where $j_{\alpha,c} : V \to M_{\alpha,c} \cong \Ult(V,F_{\alpha,c})$ and $\cp(k) > \theta^+$. Therefore $M_{\alpha,c}$ and $Z_{\alpha,c}$ share the same $(\kappa,\theta^+)-$extenders. 

\item $Z_{\alpha,c} = \K(M^2_{\alpha,c,\delta})$ and
$M^2_{\alpha,c,\delta} = Z_{\alpha,c}[\Gzero_{\alpha,x,c}*\Gone_{\alpha,c,\delta}*\Gtwo_{\alpha,c,\delta}]$, where $\Gzero_{\alpha,x,c}*\Gone_{\alpha,c,\delta}*\Gtwo_{\alpha,c,\delta}$ is $j(\pzero*\pone*\ptwo)$ generic over $Z_{\alpha,c}$, satisfying $\Gzero = \Gzero_{\alpha,x,c}\uhr \kappa+1$, $\Gone = \Gone_{\alpha,c,\delta}\uhr \kappa$, and $\Gtwo = \Gtwo_{\alpha,c,\delta}\uhr \kappa+1$.
Hence $\Gzero,\Gone,\Gtwo \in M^2_{\alpha,c,\delta}$.

\item 
$j_{\alpha,c} = i_{\alpha,c} \circ j_{\alpha}$
where $j_{\alpha} : V \to M_{\alpha} \cong \Ult(V,F_\alpha)$
and $i_{\alpha,c}$ is the $c-$derived (iterated ultrapower) embedding of $M_{\alpha}$. 
The fact that $\vec{F}$ is $\mo-$increasing implies that $\vec{F}\uhr \alpha = \la F_{\beta} \mid \beta < \alpha\ra \in M_\alpha$, which in turn, implies that
$i_c(\vec{F}\uhr\alpha) = \la F_{\beta,\c} \mid \beta < \alpha \ra$ belongs to $M_{\alpha,c}$, and thus also to $Z_{\alpha,c,} = \K(M^2_{\alpha,c,\delta})$. 
In particular $F_{\alpha',c} \in M^2_{\alpha,c,\delta}$, but $c' \geq c$ so it is clear that $F_{\alpha',c'} \in M^2_{\alpha,c,\delta}$ as well.

\item 
$\delta' \in M^2_{\alpha,c,\delta}$ since $M^2_{\alpha,c,\delta}$ is closed under
$\kappa-$sequences and $\dom(\delta') \subset \lambda \leq \kappa$. 
\end{enumerate}
Since the definition of $U^2_{\alpha',c',\delta'}$ is based on 
$F_{\alpha',c'},\Gzero,\Gone,\Gtwo$, and $\delta'$, we conclude that
$U^2_{\alpha',c',\delta'} \in M^2_{\alpha,c,\delta}$. \\
${}$\\
Suppose next that $U^2_{\alpha',\c',\delta'} \mo U^2_{\alpha,\c,\delta}$. Let
$Z_{\alpha,\c}  = L[E_{\alpha,c}] = \K(M^2_{\alpha,\c,\delta})$ and
$Z_{\alpha',\c'} = L[E_{\alpha',c'}] = \K(M^2_{\alpha',\c',\delta'})$ be the iterated ultrapowers {described in the paragraph proceeding Corollary \ref{Corollary - II - guessing omega sequences}.}
The proof of Proposition \ref{Proposition - II - MO ordering by P} shows that the fact $U^2_{\alpha',\c',\delta'} \mo U^2_{\alpha,\c,\delta}$ implies
that the $Z_{\alpha',\c'}-$side of the coiteration with $Z_{\alpha,\c}$ does not involve ultrapowers on by extenders indexed below $\theta^+$. Moreover, the coiteration must use an ultrapower by a full extender on $\kappa$, on the $Z_{\alpha,c}$.\\
From the description of $Z_{\alpha',\c'}$ and $Z_{\alpha,\c}$ we see that 
the first possible difference between $Z_{\alpha',\c'}$ and $Z_{\alpha,\c}$ is
in the normal measures on the first $\lambda$ measurable cardinals above $\kappa$, which are determined by $c'$ and $c$ respectively. Here the coiteration involves an ultrapower on the $Z_{\alpha',\c'}-$side whenever there is $j \in \lambda$ such that $\c'(j) < \c(j)$, therefore $\c' \geq \c$. 
The iteration on the $Z_{\alpha,\c}-$side will be the $(c - c')$ derived iteration whose critical points are all above $\kappa$. Let us denote the resulting iterand on the $Z_{\alpha,\c}-$side by $L[E']$. 
The next possible disagreement between $E'$ and $E_{\alpha',c'}$ is at the full $(\kappa,\theta^+)$-extenders. The $(\kappa,\theta^+)-$extenders on $E'$ and $E_{\alpha',c'}$ are $\la F_{\beta,c'} \mid \beta < \alpha\ra$ and 
$\la F_{\beta,c'} \mid \beta < \alpha'\ra$ respectively. 
These are the last extenders with critical point $\kappa$ on both sequences. As we know that the coitration must include an ultrapower on the $Z_{\alpha,c}$ side with critical point $\kappa$, we must have that  $\alpha' < \alpha$.
\end{proof}


\subsection{Separation by sets and a Final Cut - $\pX$}\label{II subsection - PX}
\begin{proposition}
The normal measures on $\kappa$ in $V^2$ are separated by sets. 
\end{proposition}
\begin{proof}
Following Definition \ref{Definition - II - psi functions},
let $\la \psi_\tau \mid \tau < \theta^+\ra \subset {}^\kappa\kappa$
be a sequence of representing functions, defined from a sequence of canonical functions  $\la \rho_{\zeta} \mid \zeta < \kappa^+\ra$ and the $\Gtwo-$derived collapsing functions $\la g_\nu : \kappa^+ \onto \Theta(\nu)^+ \mid \nu \in \Omega'\cup\{\kappa\}\ra$. We get that for every $\gamma < \theta^+$ and $U^2_{\alpha,c,\delta} \in V^2$,
$j^2_{\alpha,c,\delta}(\psi_\gamma)(\kappa) = \gamma$.\\
For every ordinal $\nu$, let $o'(\nu) < \Theta(\nu)^+$ be the length
of the maximal sequence $\vec{F_\nu} \subset E^{\K}$ of $(\nu, \Theta(\nu)^+)$ full extenders (note $\K(V^2) = V$ so $E^{\K} = E$). Note that $o'(\kappa)^{M_{\alpha}} = \alpha$ for every $\alpha < l(\vec{F})$.\\

Suppose that $\alpha < \theta^+$, $\delta : \lambda \to \theta$, and that 
 $\c : \lambda \to \omega$ is $x = \dom(\delta)$ suitable.
 Define $X_{\alpha,\c,\delta} \subset \kappa$ to be the set of $\nu \in \Omega'_x$ satisfying
 \begin{enumerate}
 \item  $o'(\nu) = \psi_\alpha(\nu)$,
 \item  $\Theta_j(\nu) = \psi_{\theta^{c(j)}_j}(\nu)$ for every $j < \lambda \cap \nu$,
 and
 \item $d(\Theta_j(\nu)) = \psi_{\delta(j)}(\nu)$ for every $j \in x \cap \nu$.
 \end{enumerate}
 
For every $U^2_{\alpha',\c',\delta'} \in V^2$ it is straightforward to verify
 $X_{\alpha,\c,\delta} \in U^2_{\alpha',c',\delta'}$ if and only if $\alpha = \alpha'$, $\c = \c'$, and  $\delta = \delta'$. 
 \end{proof}

Suppose that $X \subset \kappa$ and let $\pX$ be the final cut iteration by $X$, introduced in
\cite{OBN - Mitchell order I}, Section 7.
It is easy to see that the arguments in the proof of Lemma $7.2$ and Corollary $7.3$ apply to final cut extensions of $V^2$, therefore if $V^2[\GX]$ is a $\pX$ extension of $V^2$ then 
\begin{enumerate}
\item The normal measure on $\kappa$ in $V^2[\GX]$ are of the form $U^X$, where $U$ is a normal measure in $V^2$ with $X \not\in U$, and $U^X$ is its unique normal extension in $V^2[\GX]$.
\item For every $U^X,W^X \in V^2[\GX]$, extending $U,W \in V^2$ respectively, 
$U^X \mo W^X$ if and only if $U \mo W$.
\end{enumerate}

We apply a final cut extension to obtain a model in which $\mo(\kappa) \cong <_{R^*_{\rho,\lambda}}\uhr S$. 
By the definition of $S'$ (Definition \ref{MO II - Definition - S'} ) we know that for every element $(\alpha,c) \in S$, the set
$x = c^{-1}(\{1\})$ belongs to $S'$ and $c = c_x$ is the characteristic function of $x$. Let $\delta_c \in \prod_{j \in x} \theta^1_{j}$ be the local Prikry function, mapping each $j \in x$ to $\delta_c(j) = \theta_j < \theta^1_j = \theta^{c(j)}_j$. $\delta_c$ is clearly a local Prikry function with respect to 
$\alpha,x,c$. Thus $U^2_{\alpha,c,\delta_c}$ exists, and if
$(\alpha',c')$ is an additional element in $S$ then
$U^2_{\alpha',c',\delta_{c'}} \mo U^2_{\alpha,c,\delta}$ if and only if
$(\alpha',c') <_{R^*_{\rho,\lambda}} (\alpha,c)$. \\
As $|S| \leq \kappa$, the final cut Lemma (Lemma $7.4$ in \cite{OBN - Mitchell order I}) implies there is a set $X \subset \kappa$ in $V^2$ so that $\mo(\kappa)^{V^2[\GX]} \cong \mo(\kappa)^{V^2}\uhr \{U^2_{\alpha,c,\delta_c} \mid (\alpha,c) \in S\} \cong <_{R^*_{\rho,\lambda}}\uhr S$. \qed[Theorem \ref{MO II - main theorem}]\\

\noindent Note that small forcings of cardinality $< \kappa$ do not interrupt the construction at $\kappa$. It is therefore possible to apply the main construction 
on different cardinals to obtain a global $\mo$ behavior.
\begin{corollary} Let $V = L[E]$ is a core model.
  Suppose that for every cardinal $\lambda$ in $V$
   there is proper class of measurable cardinals $\kappa$ which carry
  a $\mo-$increasing sequence of $(\kappa,\theta^+)$ extenders $\vec{F} = \la F_\alpha \mid \alpha < \lambda\ra$, where
   \begin{itemize}
   \item $\theta = \sup_{i <\lambda}\theta_i^+$ so that $\theta_i$ is the $i-$th measurable cardinal above $\kappa$, and
   \item each $F_\alpha$ is $(\theta+1)$ strong.
   \end{itemize}
Then there is a class generic extension $V^*$ of $V$, so that every well founded order $(S,<_S)$ in $V^*$ is realized as $\mo(\kappa)$ at some measurable cardinal $\kappa$.
  \end{corollary}

\section{Open Questions}\label{Section - II - Open Problems}
 
We conclude the two-part study, presented here and in \cite{OBN - Mitchell order I}
with a few questions: \\
 
\noindent {\bf Question 1}. Is it possible to realize every well-founded order of size $\kappa^+$
as $\mo(\kappa)$? \\

\noindent {\bf Question 2}. What is the consistency strength of $\mo(\kappa) \cong R_{2,2}$, and of
 $\mo(\kappa) \cong S_{\omega,2}$? Are overlapping extenders necessary? \\

\noindent  {\bf Question 3}. What is the consistency strength of $\mo(\kappa) \cong S_{2,2}$?
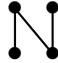
\renewcommand*{\zc}{0}%
 \renewcommand*{\oc}{1}%
\begin{center}
\begin{tikzpicture}[xscale=0.5, yscale=0.6]
    \draw  node [] at (\zc,\zc) {$\bullet$};

      \draw  node [] at (\zc,\oc) {$\bullet$};

      \draw  node [] at (\oc,\zc) {$\bullet$};

      \draw  node [] at (\oc,\oc) {$\bullet$};
 
      \draw [thick,black] (\zc,\zc) -- (\zc,\oc);      
      \draw [thick,black] (\oc,\zc) -- (\oc,\oc);      
       \draw [thick,black] (\zc,\oc) -- (\oc,\zc);       
\end{tikzpicture}
\end{center}

\noindent Note that $\rank(S_{2,2}) = 2$ but $\tamerank(S_{2,2}) = 3$, therefore the method in Part I can realize $S_{2,2}$ as $\mo(\kappa)$ from the assumption of $o(\kappa) \geq 3$.

${}$\\
${}$\\
\noindent\textbf{Acknowledgements -} \\

\noindent The author would like to express his gratitude to his supervisor Professor Moti Gitik,
 for many fruitful conversations, valuable guidance and encouragement.\\
 
\noindent The author is grateful to the referee for making many valuable suggestions and comments which 
greatly improved both the content and structure of this paper.  \\

\noindent The author would also like to thank Brent Cody and Dror Speiser for their valuable remarks and suggestions.

\raggedright

\end{document}